\documentclass[11pt,a4paper]{scrartcl}
\usepackage[top=2.5cm, bottom=2.5cm, left=2.5cm, right=2.5cm]{geometry}

\usepackage[ansinew]{inputenc}
\usepackage[T1]{fontenc}

\usepackage[pdftex]{graphicx}
\usepackage{epstopdf}
\usepackage{xr}

\usepackage[format=plain,justification=centering,singlelinecheck=true]{caption}
\usepackage{float}
\usepackage{latexsym}
\usepackage{amsmath,amssymb,amsthm,tabu,amsfonts}
\usepackage{makeidx} 
\usepackage[numbers]{natbib}
\usepackage{enumitem}
\usepackage{pdfpages}
\usepackage{color}
\usepackage{todonotes}
\usepackage{colortbl}
\usepackage{xcolor}
\usepackage{pifont}
\usepackage{booktabs}
\usepackage{tocloft} 
\usepackage{hyperref}
\usepackage{caption}
\usepackage{adjustbox}
\usepackage[font=normalsize]{subcaption}
\usepackage[hang,flushmargin,multiple,ragged]{footmisc}

\usepackage{caption}
\usepackage[font=normalsize]{subcaption}
\usepackage{accents}

\allowdisplaybreaks

\newtheorem{theorem}{Theorem}

\newtheorem{assumption}{Assumption}

\newtheorem{lemma}{Lemma}
\newtheorem{corollary}[theorem]{Corollary}
                  
\numberwithin{equation}{section}

\newcommand{\Cov}{\operatorname{Cov}}
\newcommand{\Var}{\operatorname{Var}}
\newcommand{\E}{\operatorname{E}}

\newcommand{\utilde}{\underaccent{\tilde}}

\begin{document}

\title{Sequential change point tests based on U-statistics}
\author{Claudia Kirch\footnote{\, Otto-von-Guericke University Magdeburg, Institute for Mathematical Stochastics,  Magdeburg, Germany} \and Christina Stoehr\footnote{\, Ruhr-University Bochum, Faculty of Mathematics,  Bochum, Germany; christina.stoehr@ruhr-uni-bochum.de} }
\date{December 2019}

\maketitle
\begin{center}
\begin{minipage}{0.8\textwidth}
\begin{center}\textbf{Abstract}\end{center}
We propose a general framework of sequential testing procedures based on $U$-statistics which contains as an example a sequential CUSUM test based on differences in mean but also includes  a robust sequential Wilcoxon change point procedure. Within this framework, we consider several monitoring schemes that take different observations into account to make a decision at a given time point. Unlike the originally proposed scheme that takes all observations of the monitoring period into account, we also consider a modified moving-sum-version as well as a version of a Page-monitoring scheme. The latter behave almost as good for early changes while being advantageous for later changes. For all proposed procedures we provide the limit distribution under the null hypothesis which yields the threshold to control the asymptotic type-I-error. Furthermore, we show that the proposed tests have asymptotic power one. In a simulation study we compare the performance of the sequential procedures via their empirical size, power and detection delay, which is further illustrated by means of a temperature data set.
\end{minipage}\\
\end{center}



\section{Introduction}\label{sec.intro}
Change point and data segmentation procedures have a long tradition in statistics. So called a-posteriori or offline procedures deal with the testing and segmentation of completely observed data, see e.g. \cite{aue2013} and \cite{horvath2014} for two recent survey articles mainly concerned with a-posteriori testing or \cite{fryzlewicz2014wild}, \cite{killick2012optimal}, \cite{rigaill2015pruned} and \cite{frick2014multiscale} to name but a few articles dealing with data segmentation.

On the other hand, data is collected more and more automatically with data arriving one by one, which requires a different statistical methodology designed to sequentially (online) make a decision after each new observation whether a break is likely to have occurred. In this setup an additional focus lies in the quick detection of changes after they have occurred. Examples include the monitoring of medical data of patients (e.g. \cite{fried2004online}), financial data (e.g. \cite{andreou2006monitoring}, \cite{aue2012sequential}) as well as deforestation monitoring (e.g. \cite{dutrieux2015monitoring}).

In sequential change point detection, there are different approaches: One approach aims at minimizing the mean detection delay while not causing too frequent alarms if no change occurs, see \cite{tarta} for a recent monograph.
In this paper, we follow a different approach closer related to classical testing theory, that was first proposed by
\cite{Chu}. By making use of a historic data set without changes their approach can control the asymptotic type-I-error while having power one in many situations.
This approach has been extended allowing for more general error sequences, multivariate observations but also different types of changes by several other authors including e.g.\
\cite{Hor}, \cite{aue2006change}, \cite{huvskova2005monitoring} or \cite{ciuperca2013two}. Additionally, \cite{chen2010modified}, \cite{fremdt2015page} and \cite{kirch2018modified} have proposed modified versions of the original monitoring scheme focusing more on recent monitoring observations to make a decision. These procedures outperform the original one for later changes while giving comparable results for early changes.

In this paper we present a unified theory based on $U$-statistics for different monitoring schemes. This is different from the unified theory based on estimating functions as proposed by 
\cite{Est} and further considered in \cite{kirch2018modified}. In particular, the proposed methodology includes a robust Wilcoxon monitoring, thus adapting the a-posteriori methodology proposed by \cite{Dehl}.

\subsection{Change point problem and monitoring statistics}\label{sec.model}
Following  \cite{Chu}  we assume the existence of a historic data set $X_1,\ldots, X_m$ without a change. These historic observations can be used to estimate unknown parameters consistently. Asymptotic results are then obtained by letting the length of the historic data set increase to infinity. Subsequent to the historic observations we start monitoring new incoming data by testing for a structural break after each new observation $X_{m+k},k\geq 1,$ where $k$ denotes the monitoring time, by means of a monitoring statistic $\Psi(m,k)$ not depending on future data.
We consider the following general change model
\begin{equation}\label{meanmodel}
X_{i,m}=1_{\{1\leq i\leq k^*+m\}}Y_i+ 1_{\{i>k^*+m\}}Z_{i,m},\quad i\geq 1,
\end{equation}
where $\{Y_i\}_{i\in\mathbb{Z}}$ and $\{Z_{i,m}\}_{i\in\mathbb{Z}}$ are suitable stationary time series. The distribution of the time series after the change and thus the change itself is allowed to depend on $m$. The null hypothesis then corresponds to $k^
*=\infty$. Our examples focus on the special case of a classical mean change model with $Z_{i,m}=Y_i+d_m$, i.e.
\begin{equation}\label{meanex}
X_{i,m}=Y_i+1_{\{i>k^*+m\}}d_m,\quad d_m\neq 0,
\end{equation}
where the change in the mean $d_m$ is allowed to depend on m.\\

Because the monitoring continues as long as no alarm is given, the monitoring horizon is potentially infinite. By introducing a weight function $w(m,k)$ as e.g.\ given in \eqref{exw} below it is still possible to control the asymptotic type-I-error. More precisely, an alarm is given as soon as $$w(m,k)\left|\Psi(m,k)\right|>c_{\alpha},$$ where the critical value $c_{\alpha}$ is chosen such that the testing procedure holds the level $\alpha$ asymptotically. Equivalently (for $w(m,k)\neq 0$) an alarm is given as soon as the absolute monitoring statistic $\left|\Psi(m,k)\right|$ exceeds the critical curve given by $\frac{c_{\alpha}}{w(m,k)}$. 

As long as the monitoring statistic does not exceed the critical curve, we 
 continue monitoring, such that the stopping time $\tau_m$ is given by
 \begin{align*}
				\tau_m=\begin{cases}
				\inf\{k\geq 1:w(m,k)\left|\Psi(m,k)\right|>c_{\alpha}\},\\
				\infty,\quad \mbox{if } w(m,k)\left|\Psi(m,k)\right|\leq c_{\alpha}\mbox{ for all k}.
				\end{cases}
				\end{align*}
				This setup allows to control the type-I-error as in classical statistics by choosing the critical value $c_{\alpha}$ for a given level $\alpha$ such that under the null hypothesis of no change
				$$
				\lim_{m\rightarrow\infty}P_{H_0}(\tau_m<\infty)=P_{H_0}\left(\sup_{k\geq 1}w(m,k)\left|\Psi(m,k)\right|> c_{\alpha}\right)=\alpha.
			$$
	Furthermore, it will turn out that typically under weak assumptions on the alternative such a sequential test also has asymptotic power one,				$$
				\lim_{m\rightarrow\infty}P_{H_1}(\tau_m<\infty)=1.
				$$

At each time point the question effectively comes down to a two-sample problem testing whether the distribution from the historic data set is still valid. For such a problem $U$-statistics have favorable properties, see e.g.\ \cite{ustat} for a recent monograph.  Therefore, we consider monitoring schemes that are based on the following $U$-statistic:
\begin{align}\label{detstat}
\Gamma(m,k)=\frac{1}{m}\sum_{i=1}^m\sum_{j=m+1}^{m+k}(h(X_i,X_j)-\theta),
\end{align}
where the kernel $h:\mathbb{R}^2\rightarrow\mathbb{R}$ is a measurable function and $\theta=\E(h(Y,Y_1))$ with $Y\stackrel{D}{=}Y_1$ being an independent copy of $Y_1$ (also independent of $\{Z_{i,m}\}$). The easiest example for a corresponding monitoring statistic is $\Psi(m,k)=\Gamma(m,k)$. The centering parameter $\theta$ is the expectation under the null hypothesis for independent data, while it 
still approximates $\E(h(Y_i,Y_j))$ with increasing lag $j-i$ under appropriate weak dependency assumptions.
On the other hand, under alternatives the expectation (after the change) is given by $\E(h(Y,Z_{i,m}))$ so that changes can only be detected if $\E(h(Y,Z_{i,m}))\neq\theta$. The actual magnitude of the change is then given by
\begin{align}\label{eq_Delta}
	\Delta_m=\E(h(Y,Z_{i,m}))- \theta.
\end{align}
If we allow for local changes with $\Delta_m\to 0$, Theorem~\ref{TheoremH1w2} shows that changes are detectable with asymptotic power one if
\begin{align}\label{changedecay}
	\sqrt{m}|\Delta_m|\to\infty.
\end{align}

\subsection{Examples}\label{sec.ex.kernels}
In this section, we give three examples of suitable kernels, the first two will be considered in detail in the rest of the paper.
\subsubsection{Difference-of-means (DOM) statistic}\label{sec_dom}
The first papers on sequential change point statistics in the setup of \cite{Chu} all focus on the following difference-of-means (DOM) statistic

\begin{align}\label{cusumstat}
\Gamma_{D}(m,k)=\sum_{j=m+1}^{m+k} \left(\overline{X}_m-X_j\right)
=\frac{1}{m}\sum_{i=1}^m\sum_{j=m+1}^{m+k}\left(X_i-X_j\right).\end{align}
The second representation shows that the DOM statistic is a special case of the above $U$-statistics with the DOM kernel
\begin{align}\label{kernelC}
h_{D}(x_1,x_2)=x_1-x_2,
\end{align}
such that $\theta^{D}=\E h_{D}(Y,Y_1)=0$.
Under the alternative, we obtain
\begin{align}\label{delta.cusum}
\Delta^{D}_m=\E h_{D}(Y,Y_1+d_m)-\theta^{D}=-d_m\neq 0.
\end{align}
Consequently, \eqref{changedecay} is satisfied iff $\sqrt{m}|d_m|\rightarrow\infty$ as has already been pointed out by several authors.

\subsubsection{Wilcoxon statistic}\label{sec_wil}
In the classical two-sample situation, the Wilcoxon  kernel function
\begin{align}\label{kernelW}
h_W(x_1,x_2)=1_{\{x_1<x_2\}}
\end{align}
is frequently used. In this case, if $Y_1$ is continuous with density $f_Y$ and distribution function $F_Y$, it holds
\begin{align*}
\theta^W=&\E h_W(Y,Y_1)=P(Y<Y_1) 
=\E(F_Y(Y))=\frac{1}{2}.
\end{align*}
Consequently, the Wilcoxon statistic is given by
$$\Gamma_W(m,k)=\frac{1}{m}\sum_{i=1}^m\sum_{j=m+1}^{m+k}\left(1_{\{X_i<X_j\}}-1/2\right).$$
Under the alternative, it holds for $d_m>0$
\begin{align}
	\Delta_m^{W}&=\E h_W(Y,Y_1+d_m)-\theta^W
=P(Y<Y_1)+P(Y_1\leq Y<Y_1+d_m)-\theta^W\notag\\
&=P(Y_1\leq Y<Y_1+d_m)=\int_{-\infty}^{\infty}f_Y(z)\left(F_Y(z+d_m)-F_Y(z)\right)dz>0,
\label{eq_Delta_Wil}\end{align}
because $Y$ is a continuous random variable.
An analogous result is obtained for $d_m<0$.

\subsubsection{Further examples}
As pointed out in \cite{gombay} the kernel $h(x_1,x_2)=\frac 12(x_1+x_2)$ can be used for the detection of a change from a distribution $F$ that is symmetric around zero to a distribution $G$ that is not symmetric around zero with $Y_i\sim F$ and $Z_i\sim G$. This includes location shifts away from zero but also many shifts in shape.
The expected values of the kernel function under the null hypothesis and under the alternative are given by $\theta=\E \left(\frac 12(Y+Y_1)\right)=0$ and $\theta^*=\E \left(\frac 12(Y+Z_1)\right)=\E(Z_1).$\\

The procedures proposed in this paper require $\theta$ to be known and are thus restricted to kernel functions whose expected value under the null hypothesis does not depend on the (unknown) distribution of $Y_1$. Hence, for certain change point problems it might be useful to consider differences of one-sample $U$-statistics instead which is, however, outside the scope of this paper.

\subsection{Monitoring schemes}\label{monscheme}\label{sec.scheme}
The monitoring statistic as given in 
(\ref{detstat})  with the DOM kernel is the one originally proposed by \cite{Chu}.  
Because this monitoring scheme uses all observations of the monitoring period for the comparison with the historic data set, it uses many observations still following the null distribution if changes occur rather late -- resulting in smaller power and longer detection delay.
Therefore, the following adaptations ($\Psi_{2,3}$) of the monitoring scheme have been proposed in the literature focusing more on the recent monitoring observations for the comparison:
 $$
\begin{array}{ll}
\Psi_1(m,k):=\Gamma(m,k)=\frac{1}{m}\sum_{i=1}^m\sum_{j=m+1}^{m+k}(h(X_i,X_j)-\theta)&\mbox{CUSUM},\\
&\\
\Psi_2(m,k):=\Gamma(m,k)-\Gamma(m,\lfloor mb\rfloor)=\frac{1}{m}\sum_{i=1}^m\sum_{j=m+\lfloor kb\rfloor+1}^{m+k}(h(X_i,X_j)-\theta)&\mbox{mMOSUM},\\
&\\
\Psi_3(m,k):=\sup_{0\leq l\leq k}\left|\Gamma(m,k)-\Gamma(m,l)\right|&\\
\phantom{\Gamma_3(m,k)}\,\,=\sup_{0\leq l\leq k}\left|\frac{1}{m}\sum_{i=1}^m\sum_{j=m+l+1}^{m+k}(h(X_i,X_j)-\theta)\right|&\mbox{Page-CUSUM}.
\end{array}
$$
The fixed tuning parameter $b\in (0,1)$ in the modified MOSUM (mMOSUM) determines the percentage of the earlier observations that are discarded, whereas the Page-CUSUM does not require an a priori choice of such a parameter.\\
The modified MOSUM (mMOSUM) has been proposed in \cite{chen2010modified} and the Page-CUSUM in \cite{fremdt2015page}, both for the linear model. \cite{horvath2008performance} and  \cite{aue2012reaction} consider a standard MOSUM statistic for the mean change model, however, in a comparison of all of those statistics \cite{kirch2018modified} conclude that the MOSUM procedure has noticeable power problems although its detection delay is very small for those changes that are detected.
Because we are interested in a reliable testing procedure in the sense of having a large power, we concentrate on the above three monitoring schemes in this paper.

\vspace{2mm}

The paper is organized as follows: In Section \ref{sec.as}, we derive the asymptotic theory for the above monitoring schemes. The limit distribution of the test statistics under the null hypothesis provided in Section \ref{sec.asH0} allws to control the asymptotic type-I-error. In Section \ref{sec.asH1}, we show that the proposed procedures have asymptotic power one. The asymptotic results are based on  general assumptions which are examined for specific kernels and dependency structures in Section \ref{sec.exass}. The finite sample behavior of the sequential procedures is investigated in a simulation study in Section \ref{sec.sim} where we compare different kernels and monitoring schemes. A data example is provided in Section \ref{sec.data}. Proofs are postponed to an Appendix.

\section{Asymptotics}\label{sec.as}
We provide the theory for the general $U$-statistics stated in Section \ref{sec.scheme}. First, in Section \ref{sec.asH0}, we derive the asymptotics under the null hypothesis in order to control the type-I error. Consistency of the procedures is shown in Section \ref{sec.asH1}. The assumptions on the dependency structure of the observed time series depend on the kernel function and are rather abstract. In Section~\ref{sec.exass} we derive them explicitly for both  independent observations as well as  functionals of mixing processes.

\subsection{Asymptotics under the null hypothesis}\label{sec.asH0}
In order to state the assumptions we need to introduce Hoeffding's decomposition which is a well-known tool when dealing with $U$-statistics. As proposed by \cite{Hoeff} the kernel function is decomposed under the null hypothesis as follows:
		\begin{align}\label{Hdec}
		h(x_1,x_2)=\theta+h_1(x_1)+h_2(x_2)+r(x_1,x_2)\end{align}
				with 
				\begin{align*}
				&\theta=\E(h(X_1,X_2)),\qquad
				h_1(x_1)=\E(h(x_1,X_2))-\theta,\qquad				h_2(x_2)=\E(h(X_1,x_2))-\theta,\notag\\
			&	r(x_1,x_2)=h(x_1,x_2)-h_1(x_1)-h_2(x_2)-\theta,
				\end{align*}
			
				where $X_1$ and $X_2$ are independent random variables with the same distribution.
				In particular,
				\begin{align}\label{hcenter}
				\E\left(h_1(X_1)\right)=\E\left(h_2(X_2)\right)=0.
				\end{align}
				For the DOM kernel as in Section~\ref{sec_dom} we obtain
\begin{align}\label{eq_h_dom}
&h^D_1(x_1)=x_1-E(X_2),\quad 
h^D_2(x_2)=\E(X_1)-x_2=-h^D_1(x_2),\quad r^D(x_1,x_2)=0.
\end{align}

For the Wilcoxon kernel for continuous random variables $X_1$ and $X_2$ with distribution function $F$ is given by
\begin{align}\label{eq_h_wil}
&h^W_1(x_1)=\frac 12-F(x_1),\quad h^W_2(x_2)=-h^W_1(x_2),\quad r^W(x_1,x_2)=1_{\{x_1<x_2\}}+F(x_1)-F(x_2)-\frac 12.
\end{align}

We provide the asymptotic results for a general class of weight functions with the following regularity conditions that have also been used in \cite{Est} and \cite{Weber}:
\begin{assumption}\label{regw} Let the weight function satisfy
$ $
\begin{itemize}
\item[(i)] $w(m,k)=m^{-1/2}\tilde{w}(m,k),$ where
$\tilde{w}(m,k)=\rho\left(\frac km\right)$ for $k>l_m$ with $\frac{l_m}{m}\rightarrow 0$ and $\tilde{w}(m,k)=0$ for $k\leq l_m.$ The function $\rho:(0,\infty)\rightarrow\mathbb{R}^+$ is positive and continuous.
\end{itemize}
The following two limits exist and it holds:
\begin{itemize}
\item[(ii)] $\lim_{t\rightarrow 0}t^{\gamma}\rho(t)<\infty$ for some $0\leq\gamma<\frac 12.$
\item[(iii)] $\lim_{t\rightarrow\infty}t\rho(t)<\infty.$
\end{itemize}
\end{assumption}

Part (i) allows to start the monitoring only after some observations have been collected. This is inmportant in order to avoid false alarms at the very beginning of the monitoring period. Such early false alarms often occur  as the first values of the monitoring statistic are quite volatile due to the small monitoring sample (but are asymptotically negligible). 

The regularity conditions of the weight function at zero and infinity in  (ii) and (iii) are needed to control the asymptotic behavior of the monitoring statistic at 0 and inifinity.

The following choice  fulfills these assumptions and is often used in the literature as it leads to nice limit  distributions as indicated in Corollary \ref{simple1}:
\begin{align}\label{exw}
\rho(t)=\frac{1}{1+t}\left(\frac{1+t}{t}\right)^{\gamma},
\quad 0\leq\gamma<\frac 12.
\end{align}

Instead of using explicit assumptions on the underlying dependency structure of the observed time series we work with the following high-level regularity assumptions allowing for easy future extensions to different dependency assumptions. In Section~\ref{sec.exass} we show their validity for both  independent observations as well as  functionals of mixing processes.
\begin{assumption}\label{regass}
Let $\{Y_i\}_{i\in\mathbb{Z}}$ be a stationary time series that fulfills the following assumptions for a given kernel function $h$.
\begin{enumerate}
\item[(i)]
$\E\left(\left|\sum_{i=1}^m\sum_{j=k_1}^{k_2}r(Y_i,Y_j)\right|^2\right)\leq u(m)(k_2-k_1+1)\quad\mbox{for all } m+1\leq k_1\leq k_2$\\
with  $\frac{u(m)}{m^{2-2\gamma}}\log(m)^2\rightarrow 0$ and $\gamma$ as in Assumption \ref{regw}.\\
\item[(ii)] The following functional central limit theorem holds for any $T>0$
$$\left\{\frac{1}{\sqrt{m}}\sum_{i=1}^{[mt]}\left(h_1(Y_i),h_2(Y_i)\right):0< t\leq T\right\}\stackrel{D}{\rightarrow}\left\{\left(\tilde{W}_1(t),\tilde{W}_2(t)\right):0< t\leq T\right\},$$
where $\left\{\left(\tilde{W}_1(t),\tilde{W}_2(t)\right):0< t\leq T\right\}$ is a bivariate Wiener process with mean zero and covariance matrix 
$$
\Sigma=\begin{pmatrix}
\sigma_1^2 & \rho\notag\\
\rho &\sigma_2^2
\end{pmatrix}
$$
with \begin{align}\label{defsig}
\sigma_1^2=\sum_{h\in\mathbb{Z}}\Cov(h_1(Y_0),h_1(Y_h)),\quad \sigma_2^2=\sum_{h\in\mathbb{Z}}\Cov(h_2(Y_0),h_2(Y_h)).
\end{align}
\item[(iii)] For all $0\leq\alpha<\frac{1}{2}$ the following H\'ajek-R\'enyi-type inequality holds
$$
\sup_{1\leq k\leq m}\frac{1}{m^{\frac{1}{2}-\alpha}k^{\alpha}}\left|\sum_{j=1}^{k}h_2(Y_j)\right|=O_P\left(1\right)\quad\mbox{ as }m\rightarrow\infty.
$$
\item[(iv)]The following H\'ajek-R\'enyi-type inequality holds uniformly in $m$ for any sequence $k_m>0$
$$
\sup_{k\geq k_m}\frac{1}{k}\left|\sum_{j=1}^{k}h_2(Y_j)\right|=O_P\left(\frac{1}{\sqrt{k_m}}\right)\quad\mbox{as }k_m\rightarrow\infty.
$$
\end{enumerate}
\end{assumption}
Assumption~\ref{regass} (i) is required to prove the asymptotic negligibility of the remainder term, while (ii) is a classic functional central limit theorem providing the limit distribution. Note that  $\sigma_j^2$, $j=1,2$, are the corresponding long-run variances which coincide with the variances only in the independent case. In
this case, we obtain
\begin{align}\label{eq_var_dom} 
\sigma_{1,D}^2=\sigma_{2,D}^2=\Var(Y_1)
\end{align}
with \eqref{eq_h_dom} for the DOM kernel and
\begin{align}\label{eq_var_wil} 
\sigma_{1,W}^2=\sigma_{2,W}^2=\Var\left(F(Y_1)\right)=\frac{1}{12}
\end{align}
with \eqref{eq_h_wil} for the Wilcoxon kernel if $Y_1$ is continuous with distribution function $F$.
This shows in particular, that in the independent case the Wilcoxon statistic does not require an estimator of the variance. Unfortunately, this is no longer true in the time series case as the long-run variance of the Wilcoxon kernel is not distribution-free for continuous random variables and thus requires an estimator.

The H\'ajek-R\'enyi-type inequalities in (iii) and (iv) are also standard tools in this context that hold for many dependency concepts (see e.g. Section B.1 in \cite{kdiss}), they are used in combination with Assumption~\ref{regw} (ii) -- (iii) to control the (asymptotic) behavior of the statistic at the beginning and end of the monitoring period.

\begin{theorem}\label{as.H0}
Let the regularity conditions given in Assumption \ref{regw} and \ref{regass} be fulfilled. Then, it holds under $H_0$, as $m\rightarrow\infty$,
\begin{itemize}
\item[(i)] $\sup_{k\geq 1}w(m,k)\left|\Psi_1(m,k)\right|\stackrel{\mathcal{D}}{\rightarrow}\sup_{t>0}\rho(t)\left|\sigma_2W_2(t)+t\sigma_1W_1(1)\right|$,
\item[(ii)] $\sup_{k\geq 1}w(m,k)\left|\Psi_2(m,k)\right|\stackrel{\mathcal{D}}{\rightarrow}\sup_{t>0}\rho(t)\left|\sigma_2(W_2(t)-W_2(tb))+t(1-b)\sigma_1W_1(1)\right|$,
\item[(iii)] $
\sup_{k\geq 1}w(m,k)\left|\Psi_3(m,k)\right|\stackrel{\mathcal{D}}{\rightarrow}\sup_{t>0}\rho(t)\sup_{0<s\leq t}\left|\sigma_2(W_2(t)-W_2(s))+(t-s)\sigma_1W_1(1)\right|$,
\end{itemize}
for two independent standard Wiener processes $\{W_1(t):t>0\}$ and $\{W_2(t):t>0\}$.
\end{theorem}
While the above limit distributions can already be used in combination with the Wilcoxon kernel and independent data, where the variance is completely known (see \eqref{eq_var_wil}), this is not the case for the CUSUM statistic or dependent data.

In such cases, the unknown parameters can be estimated consistently based on the historic data set, leading to the following pivotal limit distribution. Furthermore, this limit distribution can now be written as a supremum over the compact interval $[0,1]$ rather than a supremum over the positive real numbers.

\begin{corollary}\label{simple1}
Let $\{W(t):t>0\}$ be a standard Wiener processes.
\begin{itemize}
	\item[a)] If $\sigma_1=\sigma_2=:\sigma$ and $\hat{\sigma}_m\stackrel{P}{\rightarrow}\sigma$, then	
\begin{itemize}
\item[(i)] $\frac{1}{\hat{\sigma}_m}\sup_{k\geq 1}w(m,k)\left|\Psi_1(m,k)\right|\stackrel{\mathcal{D}}{\rightarrow}\sup_{0<t<1}\rho\left(\frac{t}{1-t}\right)\frac{|W(t)|}{1-t}$,
\item[(ii)] $\frac{1}{\hat{\sigma}_m}\sup_{k\geq 1}w(m,k)\left|\Psi_2(m,k)\right|\stackrel{\mathcal{D}}{\rightarrow}\sup_{0<t<1}\rho\left(\frac{t}{1-t}\right)\left|\frac{W(t)}{1-t}-(1-t(1-b))\frac{W\left(\frac{tb}{1-t(1-b)}\right)}{1-t}\right|$,
\item[(iii)] $\frac{1}{\hat{\sigma}_m}\sup_{k\geq 1}w(m,k)\left|\Psi_3(m,k)\right|\stackrel{\mathcal{D}}{\rightarrow}\sup_{0<t<1}\rho\left(\frac{t}{1-t}\right)\sup_{0<s\leq t}\left|\frac{W\left(t\right)}{1-t}-\frac{W\left(s\right)}{1-s}\right|$.
\end{itemize}
For a weight function as in \eqref{exw} these limit distributions further simplify as in this case \mbox{$\rho\left( \frac{t}{1-t} \right)=\frac{1}{t^{\gamma}}\,(1-t)$.}
%
%
\item[b)] For $\sigma_1\neq\sigma_2$ and $\hat{\sigma}_{1,m}\stackrel{P}{\rightarrow}\sigma_1$ as well as $\hat{\sigma}_{2,m}\stackrel{P}{\rightarrow}\sigma_{2}$ it holds
\begin{itemize}
\item[(i)] $\sup_{k\geq 1}\frac{\hat{\sigma}_{1,m}}{\sqrt{m}\left(\hat{\sigma}_{2,m}^2+\hat{\sigma}_{1,m}^2\frac{k}{m}\right)}\left|\Psi_1(m,k)\right|\stackrel{\mathcal{D}}{\rightarrow}\sup_{0< t< 1}\left|W(t)\right|$,
\item[(ii)] $\sup_{k\geq 1}\frac{\hat{\sigma}_{1,m}}{\sqrt{m}\left(\hat{\sigma}_{2,m}^2+\hat{\sigma}_{1,m}^2\frac{k}{m}\right)}\left|\Psi_2(m,k)\right|\stackrel{\mathcal{D}}{\rightarrow}\sup_{0<t<1}\left|W(t)-(1-t(1-b))W\left(\frac{tb}{1-t(1-b)}\right)\right|$,
\item[(iii)] $\sup_{k\geq 1}\frac{\hat{\sigma}_{1,m}}{\sqrt{m}\left(\hat{\sigma}_{2,m}^2+\hat{\sigma}_{1,m}^2\frac{k}{m}\right)}\left|\Psi_3(m,k)\right|\stackrel{\mathcal{D}}{\rightarrow}\sup_{0<t<1}\sup_{0<s\leq t}\left|W\left(t\right)-\frac{1-t}{1-s}W\left(s\right)\right|$.
\end{itemize}
\end{itemize}
\end{corollary}

For both the DOM- and the Wilcoxon kernel it holds $h_2(y)=-h_1(y)$ and thus  $\sigma_1=\sigma_2$.

\subsection{Consistency under alternatives}\label{sec.asH1}
In the following we show that the proposed procedures have asymptotic power one (for $m\to\infty$) under alternatives, i.e.\ the procedure will eventually stop under alternatives. 

In the situation of \eqref{meanmodel} we get with $\theta_m^*:=\E(h(Y,Z_{i,m})=\Delta_m+\theta$ (see \eqref{eq_Delta}) for $k>k^*$ 
\begin{align*}
\Gamma(m,k)=&\frac{1}{m}\sum_{i=1}^m\sum_{j=m+1}^{m+k^*}(h(Y_i,Y_j)-\theta)
+\frac{1}{m}\sum_{i=1}^m\sum_{j=m+k^*+1}^{m+k}(h(Y_i,Z_{j,m})-\theta_m^*)+(k-k^*)\Delta_m.
\end{align*}

The first summand is the same term as under the null hypothesis which has already been analyzed in the previous section. The second summand is a related two-sample $U$-statistic for which the two samples are no longer generated by the same distribution, while the third summand is the signal term.

In order to control the second term, consider the following version of Hoeffding's decomposition:
\begin{align}\label{hoeffh1}
h(y,z)=\theta_m^*+h^*_{1,m}(y)+h^*_{2,m}(z)+r^*_m(y,z)
\end{align}
where  we define for independent random variables $Y\stackrel{\mathcal{D}}{=}Y_1$ and $Z_m\stackrel{\mathcal{D}}{=}Z_{m,1}$
\begin{align*}
&h^*_{1,m}(y)=\E h(y,Z_m)-\theta_m^*, \quad\theta_m^*=\E h(Y,Z_m),\quad 
h^*_{2,m}(z)=h_2(z)-\Delta_m,\quad \Delta_m=\theta_m^*-\theta,
\\
&r_m^*(y,z)=h(y,z)-h^*_{1,m}(y)-h^*_{2,m}(z)-\theta_m^*.
\end{align*}
 Again, in this decomposition, $h_{1,m}^*(Y)$, $h_{2,m}^*(Z_m)$ as well as $r_m^*(Y,Z_m)$ are centered random variables.

 For the DOM kernel in model \eqref{meanex} this results in
\begin{align*}
&h^{*D}_{1,m}(y)
=y-\E(Y),\quad
h^{*D}_{2,m}(z)=
\E(Y)-z+d_m,\quad
r^{*D}(y,z)
=0.
\end{align*}
For the Wilcoxon kernel, model \eqref{meanex} and a continuous $Y$ with distribution function $F$  it holds with \eqref{eq_Delta_Wil}
\begin{align*}
&h_{1,m}^{*W}(y)
=\frac 12-F(y-d_m)-\Delta_m^W,\quad
h_{2,m}^{*W}(z)
=F(z)-\Delta_m^W+\frac 12,\quad\\
&r^{*W}(y,z)
=1_{\{y<z\}}+F(y-d_m)-F(z)+\Delta_m^W-\frac 32
\end{align*}

In order to obtain asymptotic power one, we need to impose some additional assumptions on the weight function, which are again fulfilled for the weight function  in \eqref{exw}.
\begin{assumption}\label{altrho}
\begin{itemize}
\item[(i)] If  $\frac{k^*}{m}\rightarrow\infty$, assume that $\liminf_{t\rightarrow\infty}t\rho(t)>0.$
\item[(ii)] If $\frac{k^*}{m}=O(1),$ i.e. $\frac{k^*}{m}<\nu$ for all $m\geq 1$ for some $\nu>0$, assume that there exist $t_0>\nu, \epsilon>0$ such that $\rho(t)>0$ for all $t\in (t_0-\epsilon, t_0+\epsilon)$.
\item[(iii)] If $\frac{k^*}{m}=O(1),$ i.e. $\frac{k^*}{m}<\nu$ for all $m\geq 1$ for some $\nu>0$, assume that there exist $t_0>\nu, \epsilon>0$ such that $\rho(t)>0$ for all $t\in \left(\frac{t_0}{b}-\epsilon, \frac{t_0}{b}+\epsilon\right)$.
\end{itemize}
\end{assumption}

We assume that the time series before the change fulfills the regularity conditions under the null hypothesis as given in Assumption \ref{regass}. 
For the time series after the change $\{Z_{i,m}\}$ the following weaker assumptions are required.
\begin{assumption}\label{regassH1}
Let $\{Y_i\}_{i\in\mathbb{Z}}$ and $\{Z_{i,m}\}_{i\in\mathbb{Z}}$ be stationary time series that fulfill the following assumptions for a given kernel function $h$.
\begin{enumerate}
\item[(i)]
$\sum_{i=1}^m\sum_{j=k_1}^{k_2}r_m^*(Y_i,Z_{j,m})=O_P( \sqrt{m\,(k_2-k_1+1)\,\max(m,k_2-k_1+1)})$ for all sequences $k_1(m)\le k_2(m)$.
\item[(ii)] $\frac{1}{\sqrt{m}}\sum_{i=1}^mh^*_{1,m}(Y_i)=O_p(1)$ as $m\rightarrow\infty$.
\item[(iii)] $\frac{1}{\sqrt{k_m}}\sum_{j=m+k^*+1}^{m+k^*+k_m}h^*_{2,m}(Z_{j,m})=O_p(1)$ as $k_m\rightarrow\infty$.
\end{enumerate}
\end{assumption}

The following theorem shows that the monitoring will eventually stop and the null hypothesis be rejected with asymptotic probability one under the alternative.
\begin{theorem}\label{TheoremH1w2}
	Let the regularity conditions given in Assumptions \ref{regw}, \ref{regass} (for $\{Y_i\}$) and \ref{regassH1}  be satisfied. Furthermore assume that $\sqrt{m}|\Delta_m|\rightarrow\infty$ and that Assumption~\ref{altrho} (i) holds.
	Then, it holds under the alternative $$\sup_{k\geq 1}w(m,k)\left|\Psi_j(m,k)\right|\stackrel{P}{\rightarrow}\infty$$ for $j=1,3$ if Assumption \ref{altrho} (ii) as well as for $j=2$ if Assumption~\ref{altrho} (iii)  is fulfilled.
%
\end{theorem}

\subsection{Regularity assumptions}\label{sec.exass}
In this section, we show that 
Assumption~\ref{regass} is fulfilled in the i.i.d.\ case under certain moment assumptions as well as under the weak dependency concept that has also been used in \cite{Dehl}.

We first properly define the latter:
A stochastic process $\{Z_i:i\in\mathbb{Z}\}$ is called absolutely regular (see \cite[Definition 1.1]{borovkova2001limit}) if
\begin{align*}
\beta(k)&=\sup_{i\geq 1}\left\{\E\left(\sup_{A\in\mathcal{A}_{i+k}^{\infty}}|P(A|\mathcal{A}_{-\infty}^i)-P(A)|\right)\right\}\notag\\
&=\frac 12\sup_{i\geq 1}\left\{\sup\sum_{j=1}^J\sum_{l=1}^L|P(A_j\cap B_l)-P(A_j)P(B_l)| \right\}\rightarrow 0\quad (k\rightarrow\infty)
\end{align*}
where $\mathcal{A}_{i_1}^{i_2}=\sigma(Z_{i_1},Z_{i_1+1}\ldots, Z_{i_2})$ and the inner supremum in the second representation is taken over all finite $\mathcal{A}_{-\infty}^i-$ measurable partitions $(A_1,\ldots,A_J)$ and all finite $\mathcal{A}_{i+k}^{\infty}-$ measurable partitions $(B_1,\ldots,B_L)$, $J,L$ arbitrary. A sequence $\{Y_i:i\geq 1\}$ is called an r-approximating functional (see \cite[Definition 1.4.]{borovkova2001limit}) with approximating constants $\{a_k\}_{k\geq 0}$  of $\{Z_i:i\in\mathbb{Z}\}$ if
$$\E|Y_0-\E(Y_0|Z_{-k},\ldots,Z_k)|^r\leq a_k$$
with $a_k\rightarrow 0$ as $k\rightarrow\infty$. A kernel $h:\mathbb{R}^2\rightarrow\mathbb{R}$ is called 1-continuous, if there exists a function $\Phi:(0,\infty)\rightarrow (0,\infty)$ with $\Phi(\epsilon)\rightarrow 0$ as $\epsilon\rightarrow 0$ such that for all $\epsilon>0$
\begin{align*}
&\E\left(|h(X',Y)-h(X,Y)|1_{\{|X-X'|<\epsilon\}}\right)\leq \Phi(\epsilon)\notag\\
&\E\left(|h(X,Y')-h(X,Y)|1_{\{|Y-Y'|<\epsilon\}}\right)\leq \Phi(\epsilon)
\end{align*}
for all random variables $X,X',Y$ and $Y'$ having the same marginal distribution as $Y_1,$ and such that $X,Y$ are either independent or have joint distribution $P_{(Y_1,Y_k)}$ for some integer k (see Definition 2 in \cite{Dehl}).

\begin{lemma}\label{ass.ex}
\begin{itemize}
\item[a)] Let $Y_1,Y_2,\ldots$ be i.i.d.\ random variables 
	and let $h:\mathbb{R}^2\rightarrow\mathbb{R}$ be a kernel with
\begin{align}\label{ass.sym}
Eh^2(Y_1,Y_2)<\infty,\quad
0<\E\left|h_1(Y_1)\right|^{\nu}<\infty,\quad
0<\E\left|h_2(Y_1)\right|^{\nu}<\infty
\end{align}
for some $\nu>2$. Then, Assumptions \ref{regass} are fulfilled.
\item[b)] Let $\{Y_i:i\geq 1\}$ be a 1-approximating functional with approximating constants $\{a_k\}_{k\geq 0}$ of an absolutely regular process with mixing coefficients $\{\beta(k)\}_{k\geq 0}$ and let $h(x,y)$ be a bounded 1-continuous kernel. Then, Assumptions \ref{regass} are fulfilled if
\begin{align}\label{coeff.cond}
\sum_{k\geq 1}k^2\left(\beta(k)+\sqrt{a_k}+\Phi(\sqrt{a_k})\right)<\infty.
\end{align}
\end{itemize}
\end{lemma}
The  Wilcoxon kernel is bounded and 1-continuous so that Lemma~\ref{ass.ex} b) applies.

The DOM kernel, on the other hand,  is clearly not bounded such that Lemma \ref{ass.ex} b) does not apply. However, for the DOM kernel, Assumptions~\ref{regass} can be shown for many weak dependency concepts: Note, in particular, that due to remainder term being exactly zero, Assumption~\ref{regass} (i) holds trivially for the DOM kernel. Because $h_1^D(Y_1)=Y_1-\E(Y_1)=-h_2^D(Y_1)$ (ii) comes down to a functional central limit theorem for $\{Y_i\}$, which has been derived under various assumptions on the dependency.
Similarly, the H\'ajek-R\'enyi-type inequalities in (iii) and (iv) (for $Y_i-\E(Y_i)$) have been shown under rather general assumptions. For example, as described below Assumption A.3 in \cite{Est}, all three assumptions follow from a strong invariance principle that has been derived for many processes (see e.g.\ Sections 2.4 in \cite{Auediss} for some examples).


\section{Simulation study and data analysis}
\subsection{Simulation study}\label{sec.sim}

In this section, we illustrate the finite sample performance of the sequential procedures comparing different monitoring schemes, different choices of $\gamma$ in \eqref{exw} as well as the DOM- and Wilcoxon kernel via their empirical size, power and stopping time. All empirical results are based on $10\, 000$ replications for each scenario. 

We consider the mean change model as in (\ref{meanex}) with $d_m=0.5$, $m=100$ and independent innovations $\{Y_i\}$ with a standard normal distribution as well as a standardized $t_3$ distribution, where under the alternative we consider early $k^{*}=m^{\beta}$ with $\beta=0.5$, medium late $\beta=1$ as well as late changes $\beta=1.4$. The asymptotic critical values are obtained based on $50\,000$ realizations of the limit distributions in Corollary \ref{simple1} b) where the Wiener processes are approximated on a grid of $10\, 000$ points.

For the Wilcoxon kernel, the variance as in \eqref{eq_var_wil} does not depend on the distribution of $Y_1$, while for the DOM statistics $\Var(Y_1)$ (see \eqref{eq_var_dom}) is estimated by the empirical variance of the historic data set $\widehat{\sigma}_m^2=\frac{1}{m-1}\sum_{j=1}^m\left(X_{j,m}-\frac{1}{m}\sum_{l=1}^mX_{l,m}\right)^2$.

Following \cite{kirch2018modified} we start monitoring only after $a_m=\sqrt{m}=10$ in all settings in order to avoid false positives at the very beginning of the monitoring period. This does not affect the null asymptotics.

\begin{table}[tb]
\begin{minipage}{0.48\textwidth}
\textbf{Standard normal distribution:}
\vspace{-0.3cm}
\begin{center}
\begin{adjustbox}{max width=\textwidth}
\(
\begin{tabu}{l||c|c|c||c|c|c|}
\multicolumn{1}{c||}{\textbf{Monitoring}}&\multicolumn{3}{c||}{\mbox{\textbf{DOM kernel}}}&\multicolumn{3}{c|}{\mbox{\textbf{Wilcoxon kernel}}}\\
\multicolumn{1}{c||}{\cellcolor{white}\textbf{scheme}}&\boldsymbol{\gamma=0\,\,\,\,}&\boldsymbol{\gamma=0.25}&\boldsymbol{\gamma=0.45}&\boldsymbol{\gamma=0\,\,\,\,}&\boldsymbol{\gamma=0.25}&\boldsymbol{\gamma=0.45}\\\hline
\mbox{\textbf{CUSUM}}&4.70&4.72&3.69&4.26&4.40&3.13\\\hline
\mbox{\textbf{Page-CUSUM}}&4.55&4.55&3.22&4.25&4.18&2.52\\\hline
\mbox{\textbf{mMOSUM}}&&&&&&\\
\multicolumn{1}{r||}{b=0.1}&4.62&4.83&3.61&4.35&4.48&2.91\\\hline
\multicolumn{1}{r||}{b=0.4}&4.95&5.08&3.07&4.84&4.31&2.25\\\hline
\multicolumn{1}{r||}{b=0.9}&4.90&5.28&3.94&2.09&0.86&0.03\\\hline
\end{tabu}
\)
\end{adjustbox}\\
\end{center}
\end{minipage}
\hfill
\begin{minipage}{0.48\textwidth}
\textbf{Standardized $t_3$-distribution:}
\vspace{-0.3cm}
\begin{center}
\begin{adjustbox}{max width=\textwidth}
\(
\begin{tabu}{l||c|c|c||c|c|c|}
\multicolumn{1}{c||}{\textbf{Monitoring}}&\multicolumn{3}{c||}{\mbox{\textbf{DOM kernel}}}&\multicolumn{3}{c|}{\mbox{\textbf{Wilcoxon kernel}}}\\
\multicolumn{1}{c||}{\cellcolor{white}\textbf{scheme}}&\boldsymbol{\gamma=0\,\,\,\,}&\boldsymbol{\gamma=0.25}&\boldsymbol{\gamma=0.45}&\boldsymbol{\gamma=0\,\,\,\,}&\boldsymbol{\gamma=0.25}&\boldsymbol{\gamma=0.45}\\\hline
\mbox{\textbf{CUSUM}}& 5.56& 6.87& 6.93& 4.39& 4.36& 3.12\\\hline
\mbox{\textbf{Page-CUSUM}}& 5.79& 7.12& 6.50& 4.27& 4.09& 2.37\\\hline
\mbox{\textbf{mMOSUM}}&&&&&&\\
  \multicolumn{1}{r||}{b=0.1}&6.24& 7.71& 6.86& 4.51& 4.34& 2.78\\\hline
 \multicolumn{1}{r||}{b=0.4}& 8.64&10.02& 8.41& 4.46& 3.80& 1.94\\\hline
\multicolumn{1}{r||}{b=0.9}&29.26&31.53&24.86& 2.30& 0.98& 0.03\\\hline
\end{tabu}
\)
\end{adjustbox}
\end{center}
\end{minipage}
\caption{Empirical size (in $\%$) for a nominal level of $\alpha=5\%$ for independent innovations.}
\label{empsize}

\end{table}

Table \ref{empsize} shows the empirical size. For the standard normal distribution the procedures hold the nominal level and become more and more conservative with increasing $\gamma$. This is still true for the $t_3$-distribution when using the Wilcoxon kernel whereas the DOM kernel in combination with the mMOSUM and $b=0.9$ leads to high rejection rates, which occur mainly at the beginning of the monitoring period where only few summands are included in the monitoring statistics.
While for the normal distribution, this effect can be controlled by an appropriate choice of $a_m$, for the $t_3$-distribution and large values of $b$ this is not sufficient. Hence, one should be careful when using large $b$ for heavy tailed observations.

\begin{table}[htb]
\textbf{DOM kernel:}
\vspace{-0.3cm}
\begin{center}
\begin{adjustbox}{max width=0.9\textwidth}
\(
\begin{tabu}{l||c|c|c||c|c|c||c|c|c|}
&\multicolumn{3}{c||}{\boldsymbol{\beta=0.25,\quad k^{*}=3}}&\multicolumn{3}{c||}{\boldsymbol{\beta=1,\quad k^{*}=100}}&\multicolumn{3}{c|}{\boldsymbol{\beta=1.4,\quad k^{*}=630}}\\
\multicolumn{1}{c||}{\cellcolor{white}\textbf{Monitoring scheme}}&\boldsymbol{\gamma=0}&\boldsymbol{\gamma=0.25}&\boldsymbol{\gamma=0.45}&\boldsymbol{\gamma=0}&\boldsymbol{\gamma=0.25}&\boldsymbol{\gamma=0.45}&\boldsymbol{\gamma=0}&\boldsymbol{\gamma=0.25}&\boldsymbol{\gamma=0.45}\\\hline
\mbox{\textbf{CUSUM}}&99.75&99.69&99.18&99.27&99.02&98.08&87.74&84.78&76.85\\\hline
\mbox{\textbf{Page-CUSUM}}&99.74&99.68&99.21&99.45&99.34&98.52&94.99&92.65&84.28\\\hline
\mbox{\textbf{mMOSUM}}&&&&&&&&&\\
\multicolumn{1}{r||}{b=0.1}&99.75&99.60&99.05&99.61&99.43&98.66&93.20&90.52&82.11\\\hline
\multicolumn{1}{r||}{b=0.4}&99.59&99.25&97.54&99.53&99.07&96.84&99.41&98.39&93.61\\\hline
\multicolumn{1}{r||}{b=0.9}&89.42&71.36&36.51&88.43&64.74&19.28&72.85&33.80& 6.00\\\hline
\end{tabu}
\)
\end{adjustbox}
\end{center}
$ $\\
\textbf{Wilcoxon kernel:}
\vspace{-0.3cm}
\begin{center}
\begin{adjustbox}{max width=0.9\textwidth}
\(
\begin{tabu}{l||c|c|c||c|c|c||c|c|c|}
&\multicolumn{3}{c||}{\boldsymbol{\beta=0.25,\quad k^{*}=3}}&\multicolumn{3}{c||}{\boldsymbol{\beta=1,\quad k^{*}=100}}&\multicolumn{3}{c|}{\boldsymbol{\beta=1.4,\quad k^{*}=630}}\\
\multicolumn{1}{c||}{\cellcolor{white}\textbf{Monitoring scheme}}&\boldsymbol{\gamma=0}&\boldsymbol{\gamma=0.25}&\boldsymbol{\gamma=0.45}&\boldsymbol{\gamma=0}&\boldsymbol{\gamma=0.25}&\boldsymbol{\gamma=0.45}&\boldsymbol{\gamma=0}&\boldsymbol{\gamma=0.25}&\boldsymbol{\gamma=0.45}\\\hline
\mbox{\textbf{CUSUM}}&99.68&99.59&99.00&99.10&98.84&97.79&86.52&83.59&74.78\\\hline
\mbox{\textbf{Page-CUSUM}}&99.66&99.55&99.02&99.39&99.13&98.20&94.68&91.65&82.25\\\hline
\mbox{\textbf{mMOSUM}}&&&&&&&&&\\
\multicolumn{1}{r||}{b=0.1}&99.64&99.52&98.86&99.54&99.33&98.48&92.40&89.30&80.67\\\hline
\multicolumn{1}{r||}{b=0.4}&99.45&99.10&97.62&99.39&98.92&96.80&99.21&98.16&93.47\\\hline
\multicolumn{1}{r||}{b=0.9}&91.58&80.55&52.72&91.00&76.98&38.45&76.50&44.73& 9.57\\\hline
\end{tabu}
\)
\end{adjustbox}
\end{center}
\caption{Size corrected power (in $\%$) for independent $N(0,1)$-distributed innovations.}
\label{emppower}
\vspace{0.5cm}

%
\textbf{DOM kernel:}
\vspace{-0.3cm}
\begin{center}
\begin{adjustbox}{max width=0.9\textwidth}
\(
\begin{tabu}{l||c|c|c||c|c|c||c|c|c|}
\multicolumn{1}{c||}{\mbox{\textbf{Monitoring}}}&\multicolumn{3}{c||}{\boldsymbol{\beta=0.25,\quad k^{*}=3}}&\multicolumn{3}{c||}{\boldsymbol{\beta=1,\quad k^{*}=100}}&\multicolumn{3}{c|}{\boldsymbol{\beta=1.4,\quad k^{*}=630}}\\
\multicolumn{1}{c||}{\mbox{\textbf{scheme}}}&\boldsymbol{\gamma=0}&\boldsymbol{\gamma=0.25}&\boldsymbol{\gamma=0.45}&\boldsymbol{\gamma=0}&\boldsymbol{\gamma=0.25}&\boldsymbol{\gamma=0.45}&\boldsymbol{\gamma=0}&\boldsymbol{\gamma=0.25}&\boldsymbol{\gamma=0.45}\\\hline
\mbox{\textbf{CUSUM}}&98.48&97.96&96.66&97.87&97.08&94.78&88.33&84.07&74.53\\\hline
\mbox{\textbf{Page-CUSUM}}&98.39&97.73&96.17&98.01&97.12&94.85&92.82&88.49&78.21\\\hline
\mbox{\textbf{mMOSUM}}&&&&&&&&&\\
\multicolumn{1}{r||}{b=0.1}&98.26&97.64&95.94&98.22&97.43&95.33&91.62&87.66&78.42\\\hline
\multicolumn{1}{r||}{b=0.4}&97.17&95.48&89.94&97.12&95.30&89.02&96.43&93.80&84.27\\\hline
  \multicolumn{1}{r||}{b=0.9}&26.30&16.77&12.13&25.03&13.99& 9.03&14.42& 7.78& 7.45\\\hline
\end{tabu}
\)
\end{adjustbox}
\end{center}
$ $\\
\textbf{Wilcoxon kernel:}
\vspace{-0.3cm}
\begin{center}
\begin{adjustbox}{max width=0.9\textwidth}
\(
\begin{tabu}{l||c|c|c||c|c|c||c|c|c|}
\multicolumn{1}{c||}{\mbox{\textbf{Monitoring}}}&\multicolumn{3}{c||}{\boldsymbol{\beta=0.25,\quad k^{*}=3}}&\multicolumn{3}{c||}{\boldsymbol{\beta=1,\quad k^{*}=100}}&\multicolumn{3}{c|}{\boldsymbol{\beta=1.4,\quad k^{*}=630}}\\
\multicolumn{1}{c||}{\mbox{\textbf{scheme}}}&\boldsymbol{\gamma=0}&\boldsymbol{\gamma=0.25}&\boldsymbol{\gamma=0.45}&\boldsymbol{\gamma=0}&\boldsymbol{\gamma=0.25}&\boldsymbol{\gamma=0.45}&\boldsymbol{\gamma=0}&\boldsymbol{\gamma=0.25}&\boldsymbol{\gamma=0.45}\\\hline
\mbox{\textbf{CUSUM}}&100.00&100.00&100.00&100.00&100.00&100.00& 99.69& 99.37&97.93\\\hline
\mbox{\textbf{Page-CUSUM}}&100.00&100.00&100.00&100.00&100.00&100.00&100.00& 99.99&99.86\\\hline
\mbox{\textbf{mMOSUM}}&&&&&&&&&\\
\multicolumn{1}{r||}{b=0.1}&100.00&100.00&100.00&100.00&100.00&100.00& 99.93& 99.89&99.25\\\hline
\multicolumn{1}{r||}{b=0.4}&100.00&100.00&100.00&100.00&100.00& 99.99&100.00&100.00&99.98\\\hline
\multicolumn{1}{r||}{b=0.9}& 99.81& 98.82& 88.25& 99.77& 98.56& 83.37& 99.24& 92.86&50.94\\\hline
\end{tabu}
\)
\end{adjustbox}
\end{center}
\caption{Size corrected power (in $\%$) for independent $t_3$-distributed innovations.}
\label{emppower.t}
\end{table}

\begin{center}
\begin{figure}[tb]
\begin{tabular}{llll}
\hspace{-0.3cm}\vspace{-0.5cm}&\multicolumn{1}{c}{$\beta=0.25$}&\multicolumn{1}{c}{$\beta=0.75$}&\multicolumn{1}{c}{$\beta=1.4$}\\
\multicolumn{1}{l}{\raisebox{18ex}[0ex][0ex]{\small$\gamma=0$}}&\hspace{-0.3cm}\includegraphics[width=0.3\textwidth]{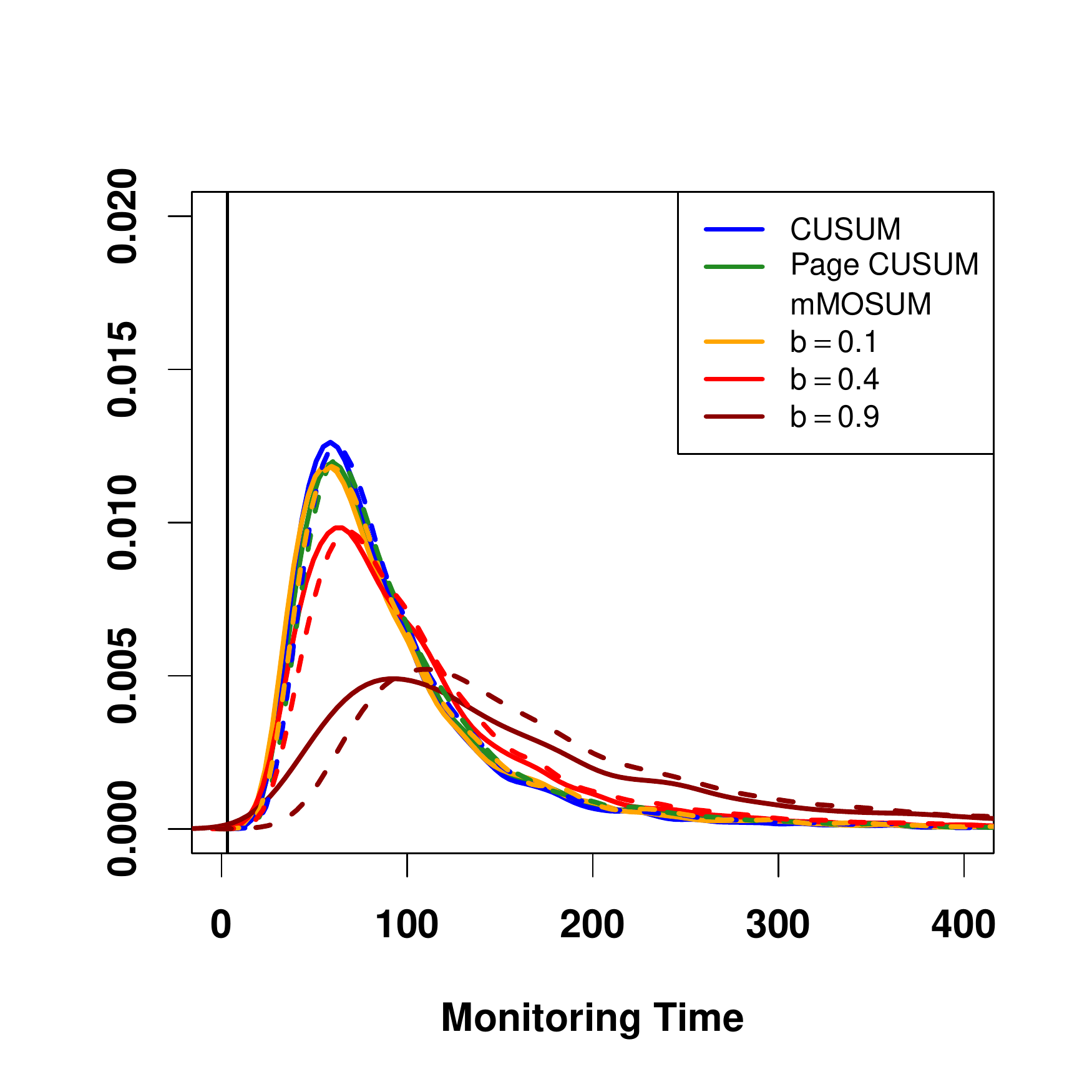}
&\hspace{-0.3cm}\includegraphics[width=0.3\textwidth]{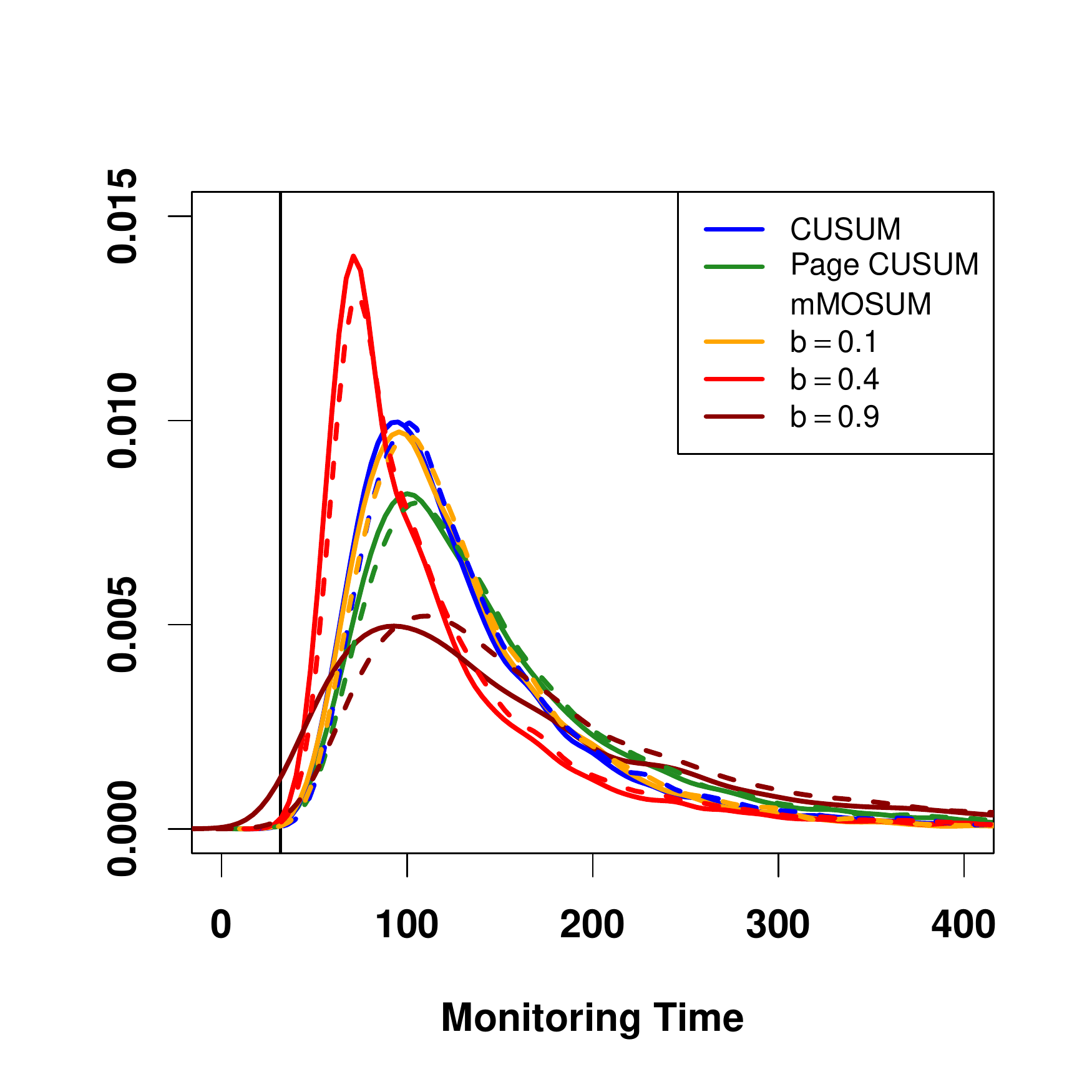}
&\hspace{-0.3cm}\includegraphics[width=0.3\textwidth]{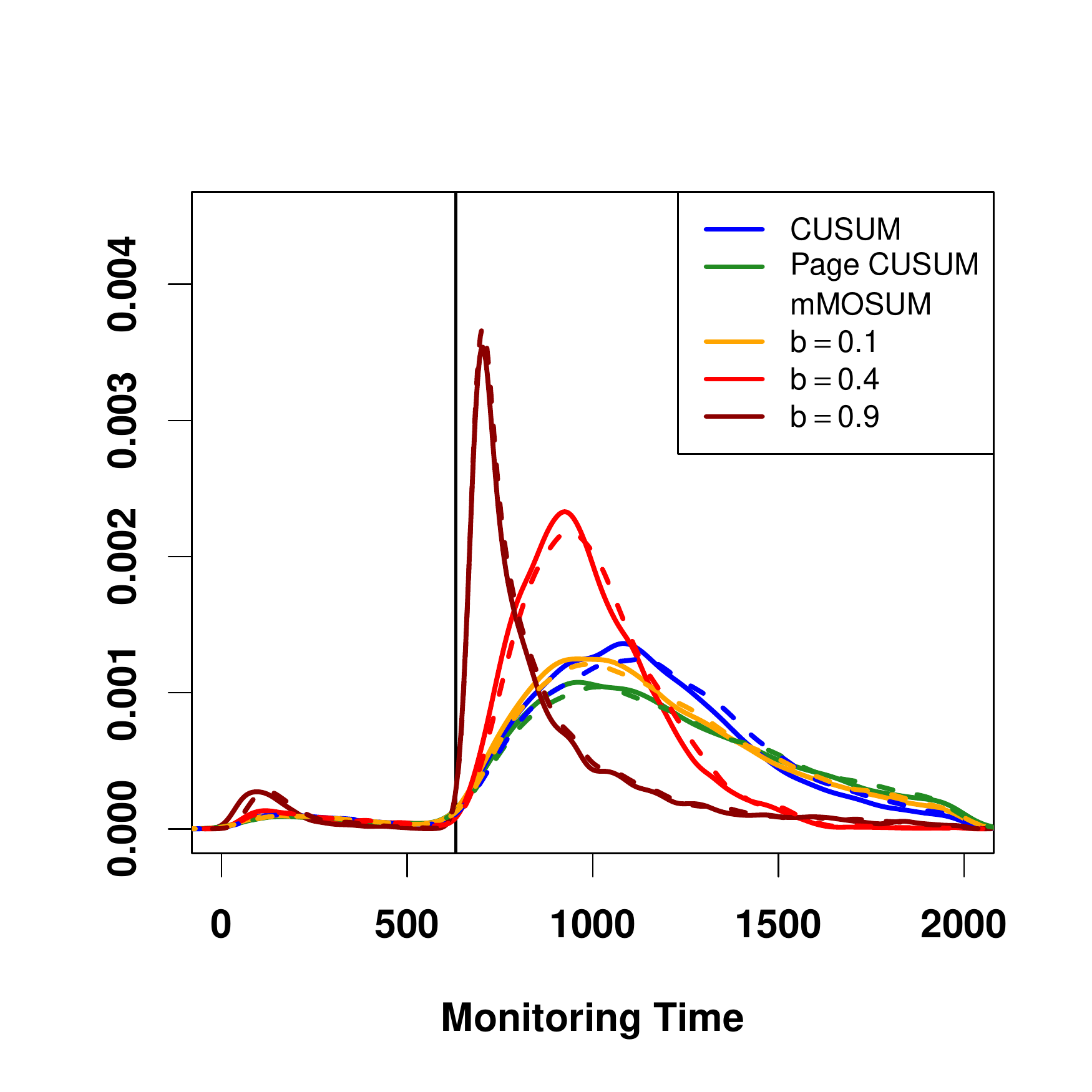}\\
\multicolumn{1}{l}{\raisebox{18ex}[0ex][0ex]{\small$\gamma=0.45$}}&\hspace{-0.3cm}\includegraphics[width=0.3\textwidth]{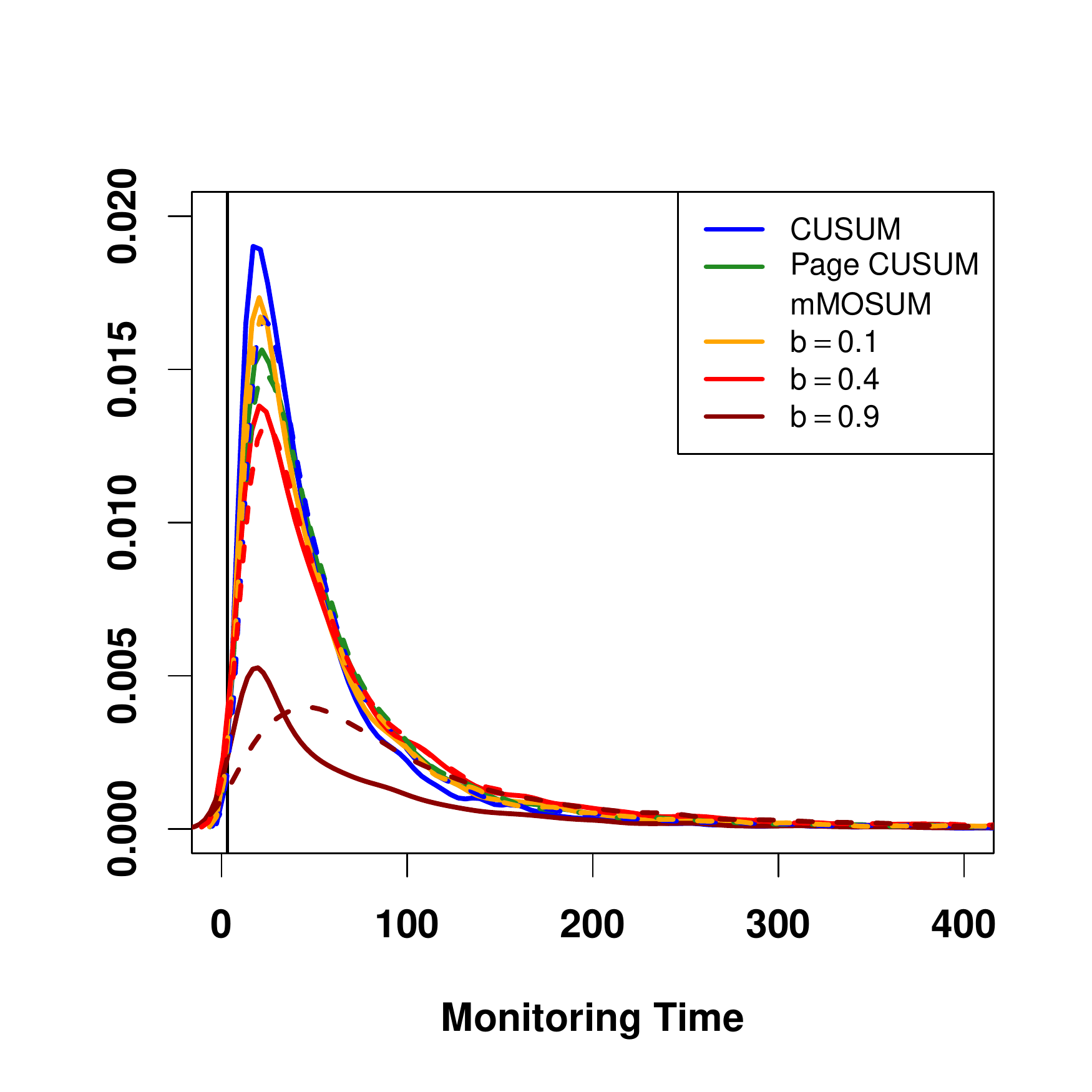}
&\hspace{-0.3cm}\includegraphics[width=0.3\textwidth]{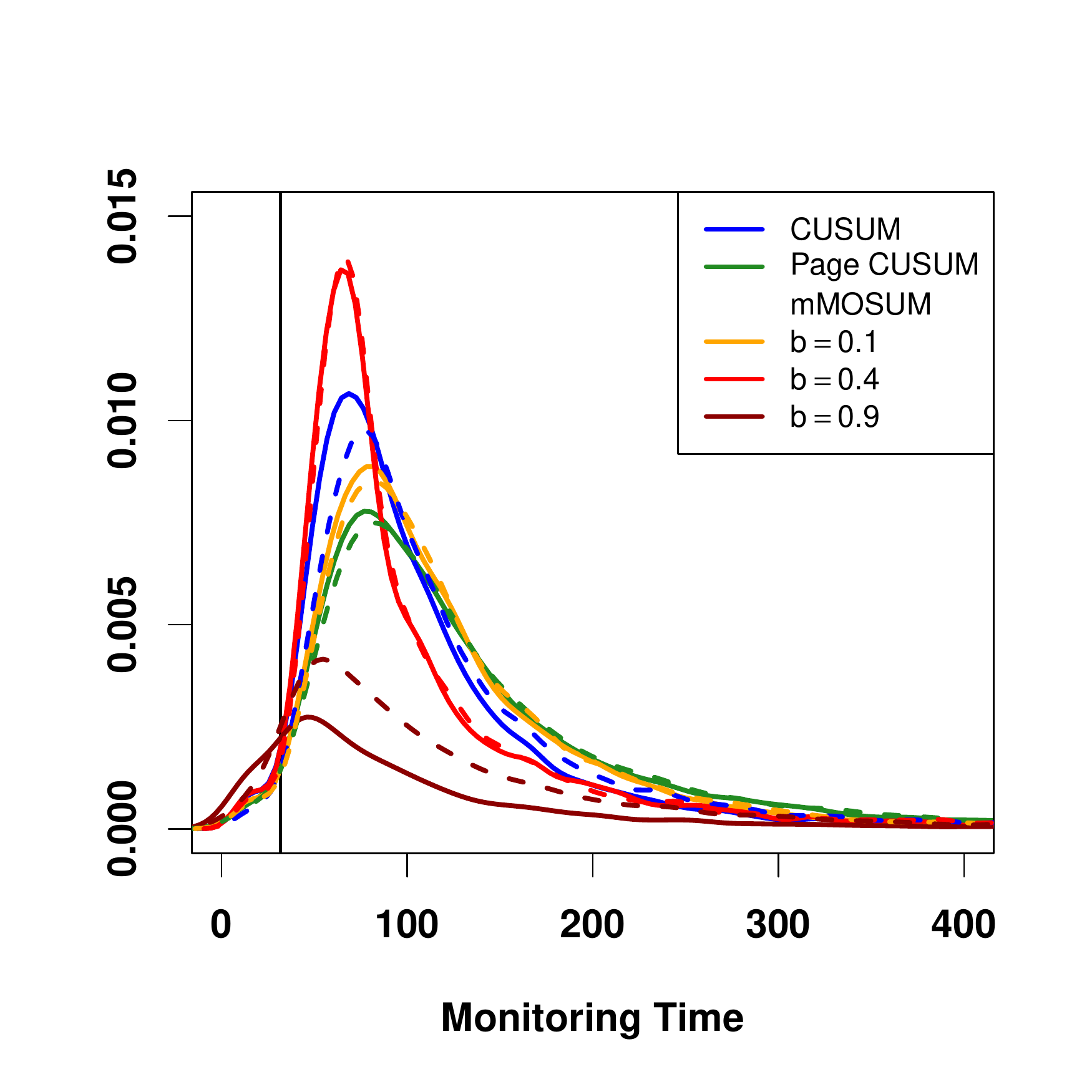}
&\hspace{-0.3cm}\includegraphics[width=0.3\textwidth]{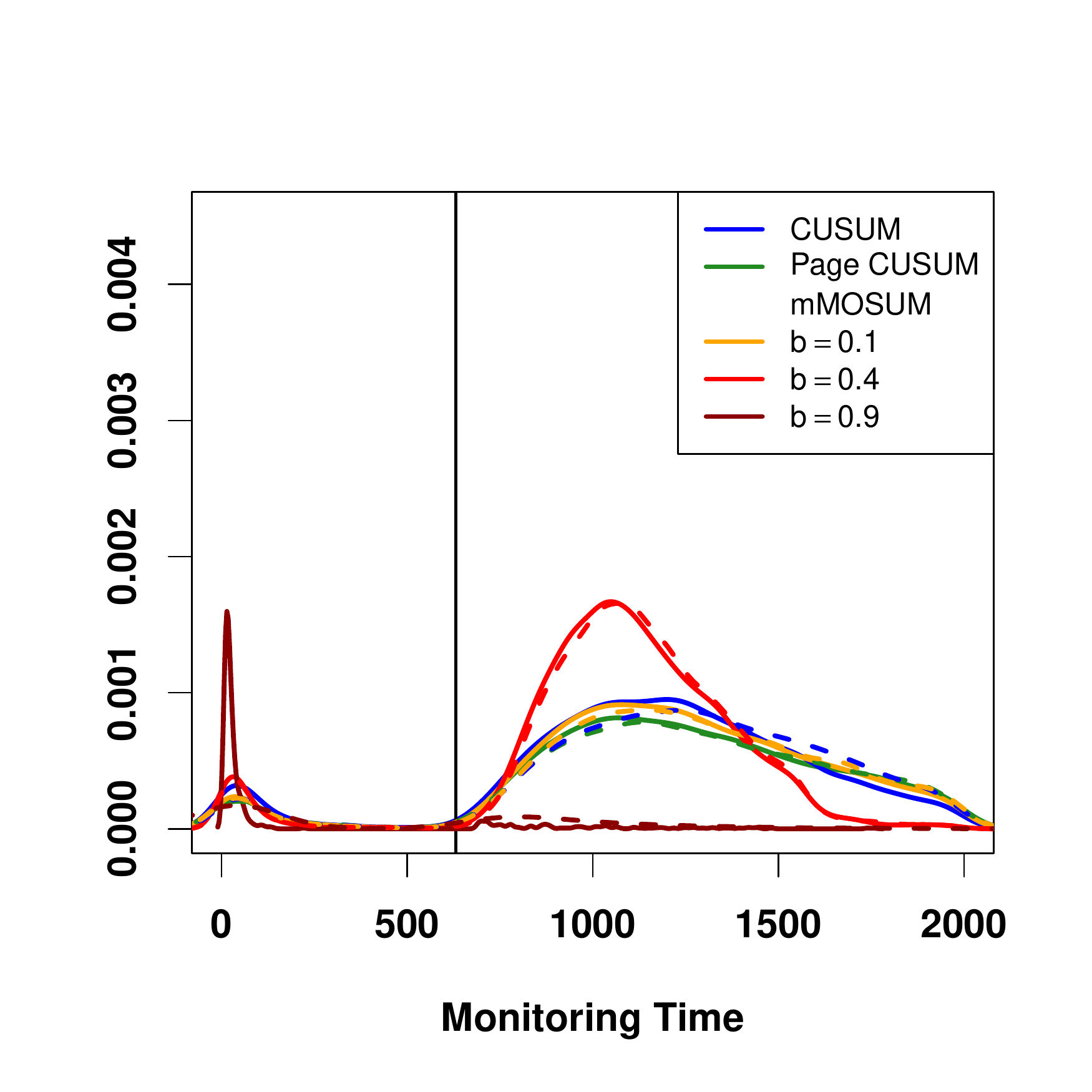}\\
\end{tabular}
	\caption{Estimated densities of the stopping time for the DOM kernel (solid lines) and the Wilcoxon kernel (dashed lines) for independent $N(0,1)$-distributed observations.}
	\label{run1}
\vspace{0.2cm}
\begin{tabular}{llll}
\hspace{-0.3cm}\vspace{-0.5cm}&\multicolumn{1}{c}{$\beta=0.25$}&\multicolumn{1}{c}{$\beta=0.75$}&\multicolumn{1}{c}{$\beta=1.4$}\\
\multicolumn{1}{l}{\raisebox{18ex}[0ex][0ex]{\small$\gamma=0$}}&\hspace{-0.3cm}\includegraphics[width=0.3\textwidth]{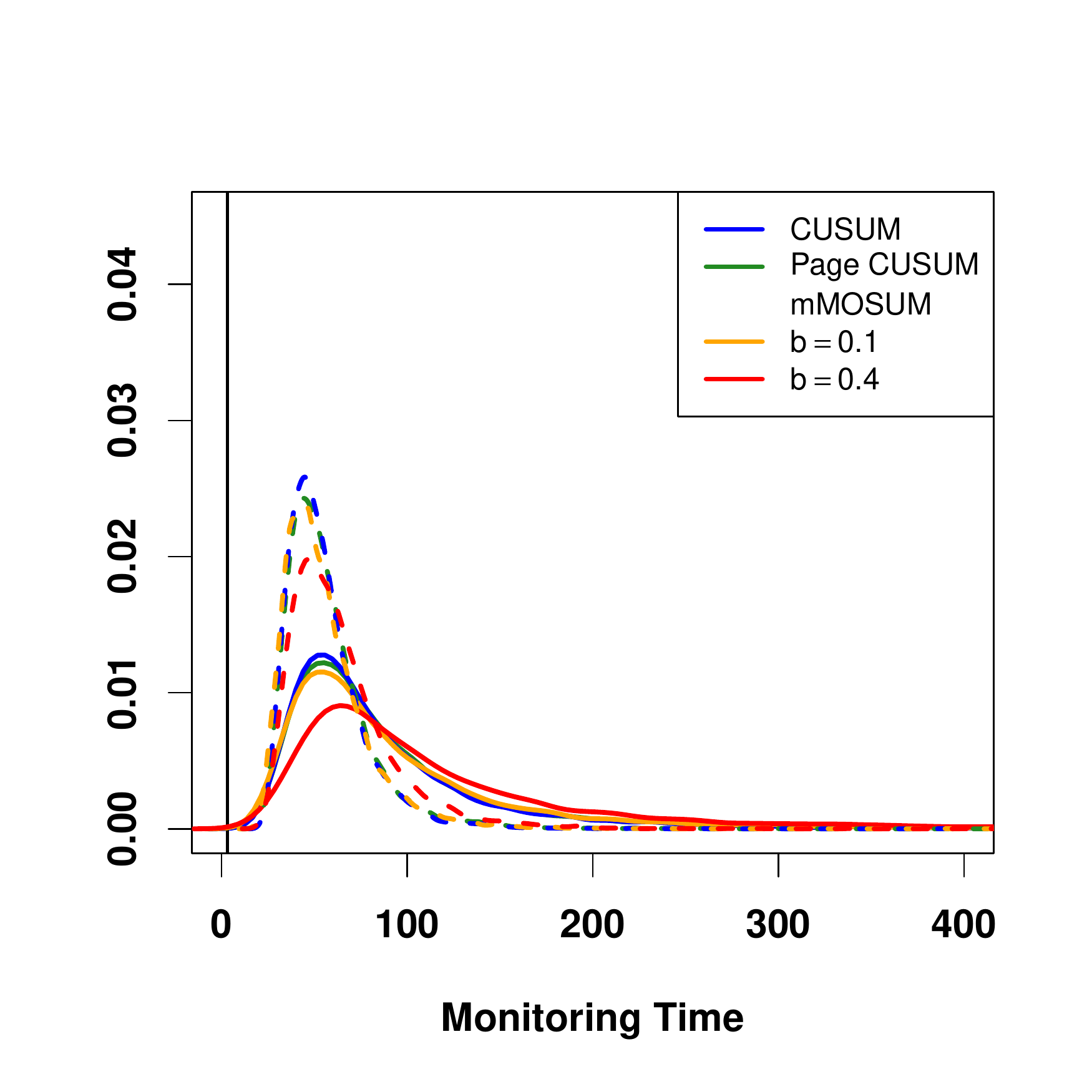}
&\hspace{-0.3cm}\includegraphics[width=0.3\textwidth]{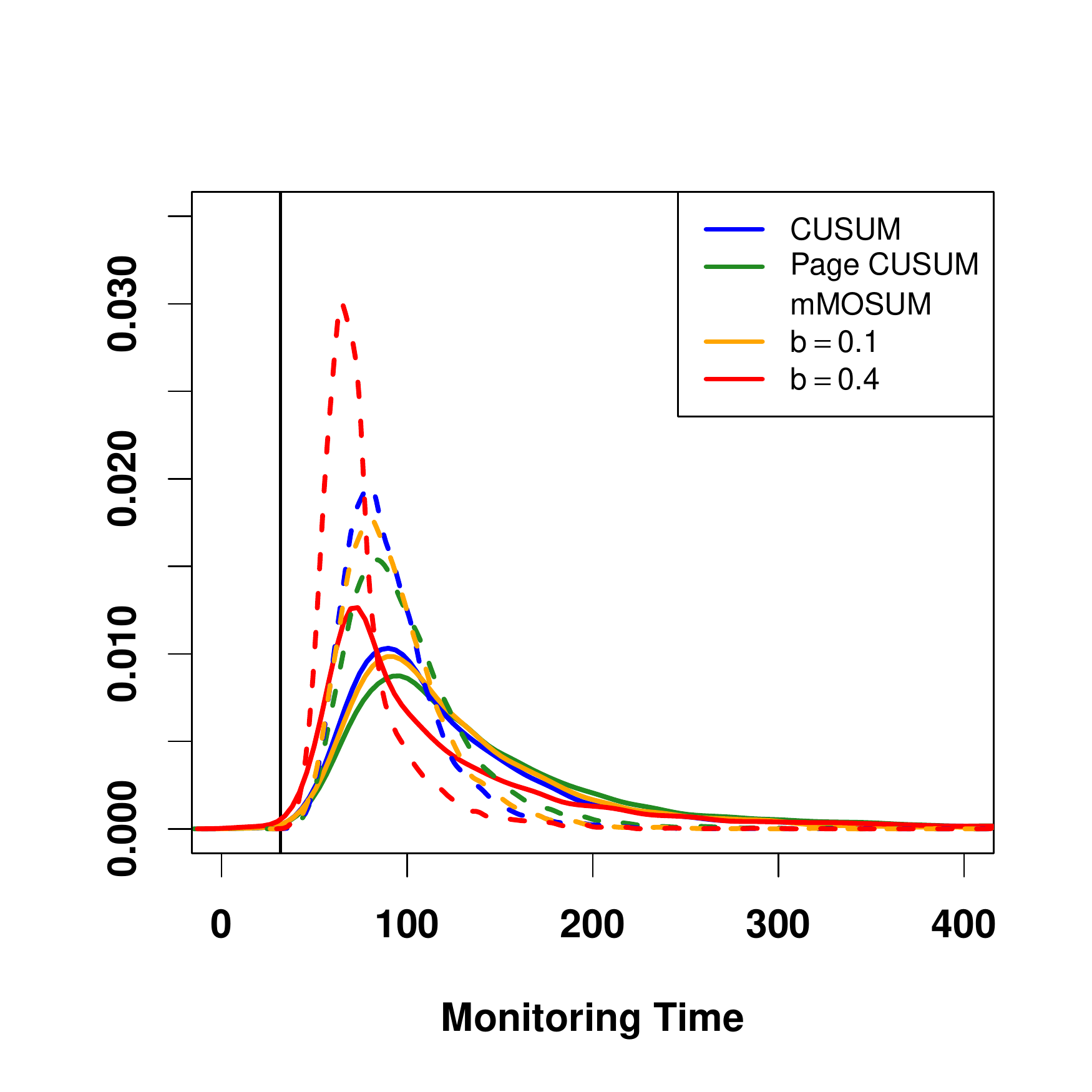}
&\hspace{-0.3cm}\includegraphics[width=0.3\textwidth]{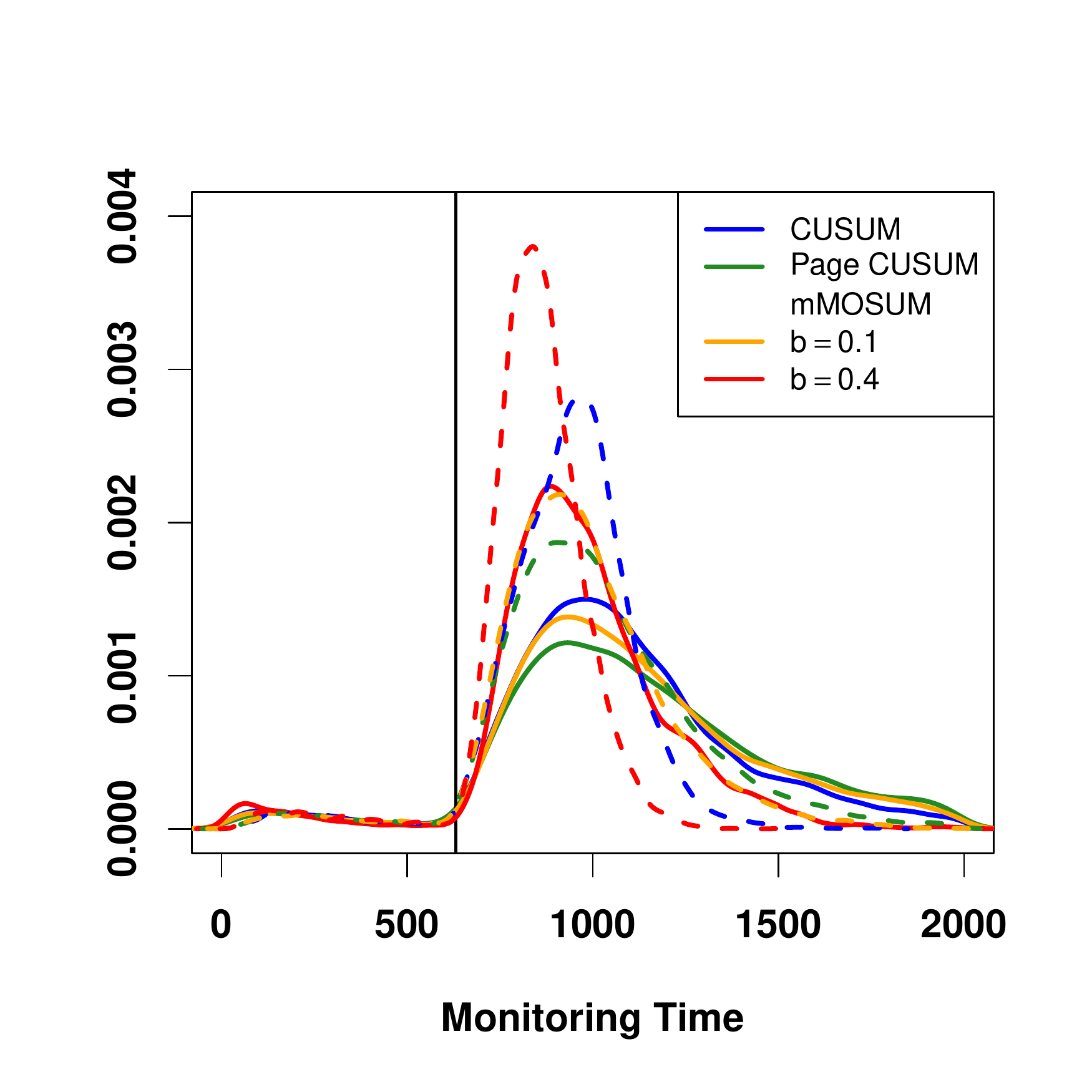}\\
\multicolumn{1}{l}{\raisebox{18ex}[0ex][0ex]{\small$\gamma=0.45$}}&\hspace{-0.3cm}\includegraphics[width=0.3\textwidth]{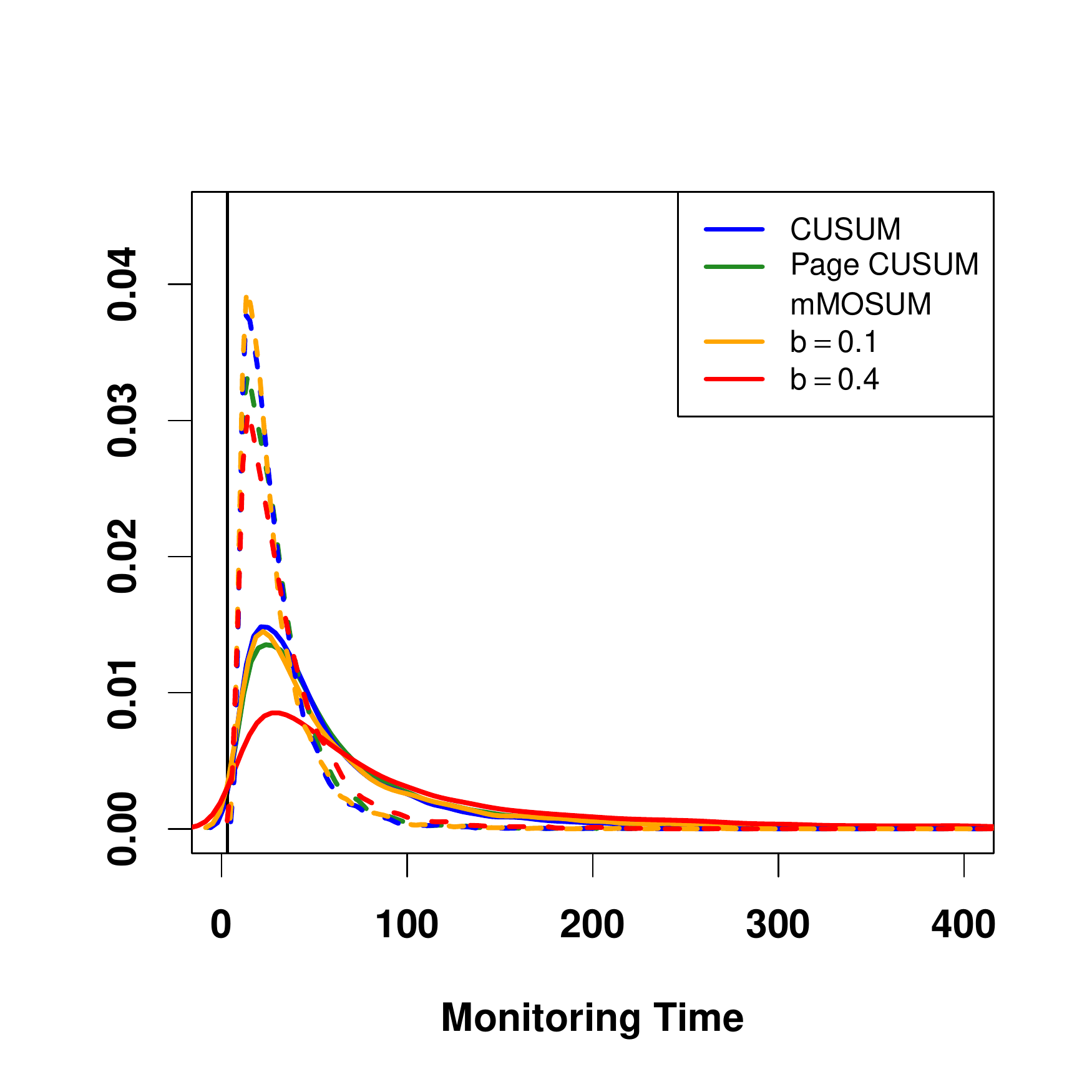}
&\hspace{-0.3cm}\includegraphics[width=0.3\textwidth]{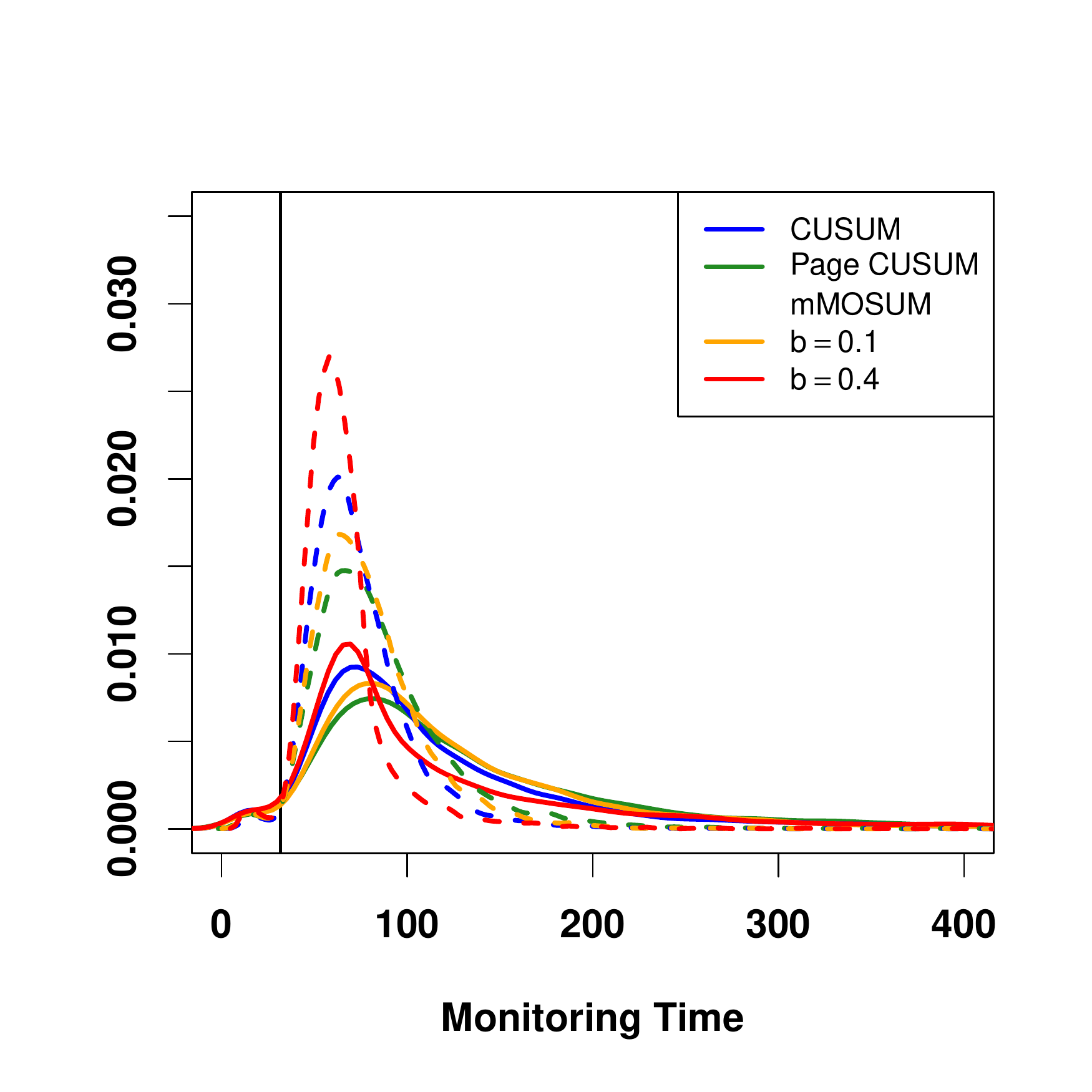}
&\hspace{-0.3cm}\includegraphics[width=0.3\textwidth]{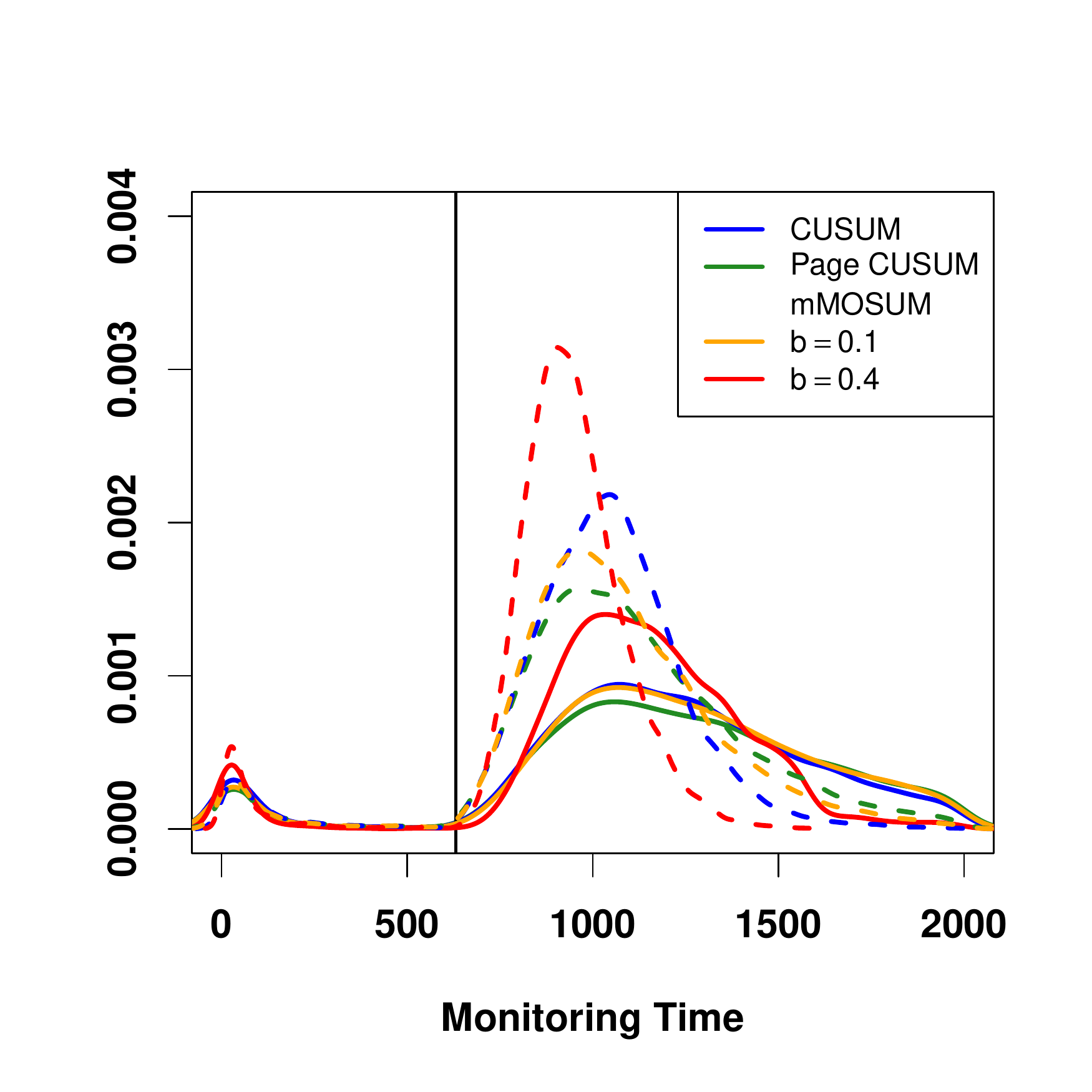}\\
\end{tabular}
	\caption{Estimated densities of the stopping time for the DOM kernel (solid lines) and the Wilcoxon kernel (dashed lines) for independent $t_3$-distributed observations.}
	\label{run1.t}
	\end{figure}
	\end{center}
	
	Table \ref{emppower} reports the size-corrected power (i.e.\ the power that is achieved if the critical values are chosen such that all procedures have empirical size $\alpha$) while Figures~\ref{run1} and \ref{run1.t} show a density plot of the size-corrected run length that is scaled in such a way that the area under the curve integrates to the empirical power.

As has already been observed for the CUSUM statistic in \cite{kirch2018modified}   $\gamma=0$ has the highest power in all situations, where the superiority of $\gamma=0$ is particularly strong for late changes ($\beta=1.4$). In terms of detection delay the CUSUM statistic is almost always outperformed by the other monitoring schemes even for early changes.
 The modified MOSUM with $b=0.9$ has a very poor power when using $\gamma=0.45$, in particular for late changes. The peak at the beginning of the monitoring period in Figure \ref{run1} seems somewhat surprising as we already wait for $a_m=10$ observations before we start monitoring. However, it turns out that it relates to only $5\%$ false positives out of all simulation runs.
The modified MOSUM with $b=0.4$ has the most stable power with respect to the time of the change for both kernels. It also has generally the smallest detection delay with the exception of very late changes where it is outperformed by the modified MOSUM with $b=0.9$ and $\gamma=0$.

For early ($\beta=0.25$) and medium-late changes ($\beta=1$) the detection delay for $\gamma=0.45$ is somewhat smaller than for $\gamma=0$ despite the fact that the power is smaller in this case.
This indicates that early changes are detected either very quickly or not at all for $\gamma\neq 0$, whereas for $\gamma=0$ it might take a bit longer but the change will be detected at some point with a very high reliability.

Furthermore, the Wilcoxon kernel almost achieves the same power and detection delay as the DOM kernel for normal data, but clearly outperforms the latter in case of more heavy-tailed $t$-innovations (with $3$ degrees of freedom). 
An additional benefit of the Wilcoxon kernel is its superiority with respect to robustness against outliers as can clearly be seen in Table 6.5 in \cite{diss} and in the density plots of the stopping times in Figure 6.3 in \cite{diss}.

\clearpage
\subsection{Data analysis}\label{sec.data}

\begin{figure}[tb]
Monthly \textbf{mean} temperature:
\begin{center}
\vspace{-0.5cm}
\includegraphics[width=0.8\textwidth]{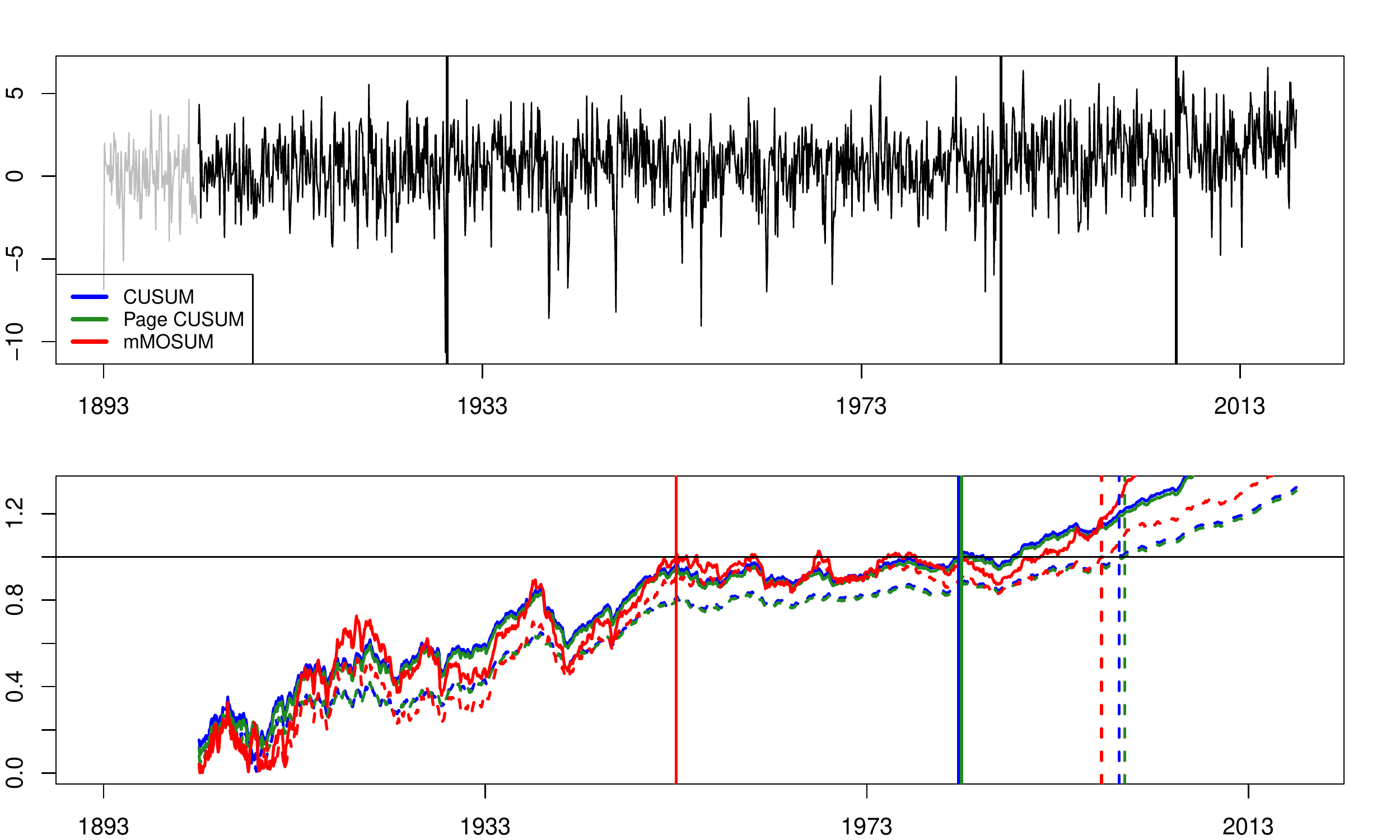}
\end{center}
%
Monthly \textbf{minimal} temperature:
\vspace{-0.5cm}
\begin{center}
\includegraphics[width=0.8\textwidth]{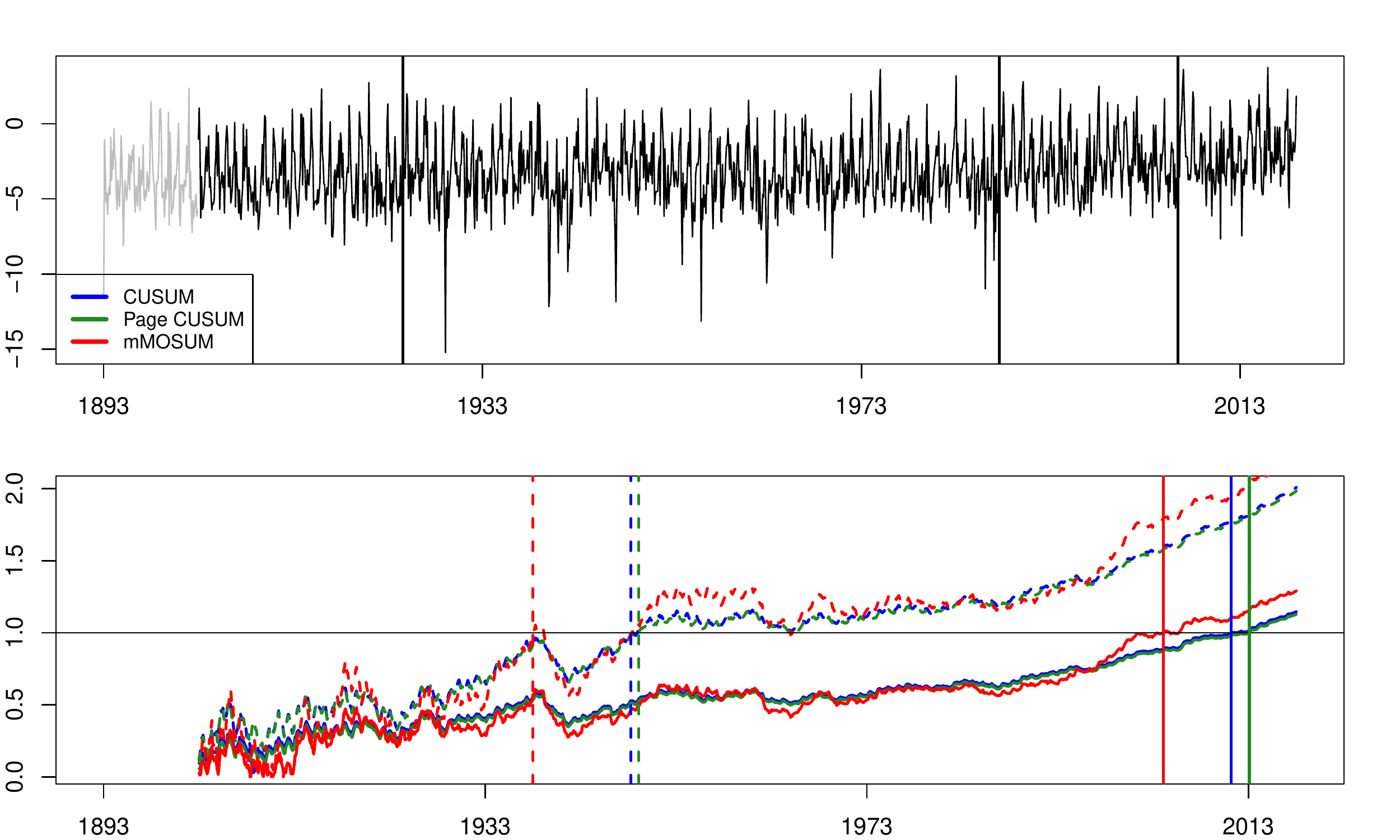}
\end{center}
\caption{Deseasonalized monthly \textbf{mean} and \textbf{minimal} temperature in Potsdam, Germany. In both cases, the upper plot  shows the data, as well as the two change points detected by an a posteriori procedure, where the second one has been dominating. The lower plot shows the normalized monitoring statistics where monitoring is stopped as soon as the horizontal line at one is crossed. The vertical lines indicate when this happens. The Wilcoxon monitoring statistics are indicated by the dashed line, the DOM by a solid line. The blue line indicates the CUSUM monitoring, green Page CUSUM and red mMOSUM.}
	\label{mintemp}
	\end{figure}

In the following data example we analyze the monthly mean and minimal temperatures in Potsdam, Germany \footnote{provided by DWD at {https://opendata.dwd.de/climate\_environment/CDC/observations\_germany/climate/monthly/kl/historical/}} from 1893 to 2018. We use the first 120 months as our historic data set, thus starting the monitoring in 1913. In a preprocessing step we remove seasonalities by substracting the average temperature by month obtained from the historic data set. We then apply the different monitoring schemes with the DOM- as well as the Wilcoxon kernel to both data sets using the weight function in \eqref{exw} with $\gamma=0$. We test at a level of $\alpha=5\%$ and use $b=0.4$ for the modified MOSUM procedures.  For a better comparison the monitoring statistics are divided by the respective critical value and the weight function such that the null hypothesis is rejected as soon as the normalized monitoring statistic crosses the horizontal line at 1. The resulting rejection times as well as the normalized monitoring statistics can be seen in Figure~\ref{mintemp}. The black lines represent the change points resulting from the offline Wilcoxon procedure provided in \cite{Dehl}. First, the changes in October 1987 (mean temperatures) respectively in August 1987 (minimal temperatures) have been detected with p-values less than $5 \cdot 10^{-7}$. Splitting the data sets in those points and reapplying the procedure reveals additional changes in the first part of the data sets in April 1929/ August 1924 and in the second part in May/ June 2006 with p-values between $0.0004$ and $0.094$. For the monthly mean temperatures the mMOSUM in combination with the DOM kernel has a much smaller detection delay than all other procedures. Furthermore, the sequential procedures based on the DOM kernel detect the first change whereas the procedures based on the Wilcoxon kernel only reject after the second and more significant change has appeared. Opposite behavior, i.e. superiority of the Wilcoxon kernel, with an even stronger effect can be observed when analyzing the monthly minimal temperatures. 
\vspace{2mm}

In conclusion, the present data example confirms the findings from the simulation study: On the one hand, the DOM kernel performs better for the mean monthly temperatures which are closer to normality. On the other hand, the Wilcoxon kernel is advantageous for the minimal monthly temperatures which are further away from normality and do contain more outliers. Additionally, for a given kernel the modified MOSUM has the smallest detection delay in all cases.

\clearpage
\appendix
	
\section{Proofs}
We 
first prove a preliminary lemma  that will be used to show the negligibility of the remainder term in the respective Hoeffding decompositions and is also of independent interest. A more detailed proof can be found in the proof of Lemma 4.1 in \cite{diss}.

\begin{lemma}\label{r.conv.generalg}
Let $\{Y_i\}_{i\geq 1}$ and $\{Y_{i,m}'\}_{i\geq 1}$ be sequences of random variables. Let Assumption \ref{regw} be fulfilled for the weight function. Assume that for $g_m:\mathbb{R}^2\rightarrow \mathbb{R}$ it holds
\begin{equation}\label{as1H0w}
\E\left(\left|\sum_{i=1}^m\sum_{j=k_1}^{k_2}g_m(Y_i,Y_{j,m}')\right|^2\right)\leq u(m)(k_2-k_1+1)\quad\mbox{for all } 1\leq k_1\leq k_2
\end{equation}
with $\frac{u(m)}{m^{2-2\gamma}}\log^2(m)\rightarrow 0$ for all $\delta> 0$. Then, it holds as $m\rightarrow\infty$
$$\sup_{k\geq 1}w(m,k)\sup_{0\leq l\leq k}\left|\frac{1}{m}\sum_{i=1}^m\sum_{j=l+1}^{k}g_m(Y_i,Y_{j,m}')\right|=o_P(1).$$
\end{lemma}

\begin{proof}
First, we show that the  supremum over $1\le k\le m$ converges to zero.
Since the double sum in \eqref{as1H0w} equals $\sum_{j=k_1}^{k_2}\xi_m(j)$ with $\xi_m(j)=\sum_{i=1}^mg_m(Y_i,Y_{j,m}')$, by \eqref{as1H0w} Theorem 3 in \cite{moricz1976} in combination with the Markov inequality yields
\begin{align*}
	\max_{1\leq k\leq m}\left|\sum_{i=1}^m\sum_{j=1}^{k}g_m(Y_i,Y_{j,m}')\right|=O_P\left(\sqrt{m\,u(m)\,\log^2(m)}\right)=o_P(m^{3/2-\gamma})
\end{align*}
as $u(m)\log^2(m)/m^{2-2\gamma}\to 0$.
Additionally, by Assumption~\ref{regw} (i) and (ii) it holds $\max_{1\le k\le m} w(m,k)=O(m^{\gamma-1/2})$, from which the negligibility of the supremum over $1\le k\le m$ follows as 
\begin{align}\label{eq_ck_lem2}
	&\sup_{1\le k\le m}w(m,k)\sup_{0\leq l\leq k}\left|\frac{1}{m}\sum_{i=1}^m\sum_{j=l+1}^{k}g_m(Y_i,Y_{j,m}')\right|\le \frac{2}{m^{3/2-\gamma}}\,\max_{1\leq k\leq m}\left|\sum_{i=1}^m\sum_{j=1}^{k}g_m(Y_i,Y_{j,m}')\right|=o_P(1).
\end{align}

For the supremum over $k>m$ first note that by a version of above theorem in \cite{moricz1976} (for the exact details we refer to Corollary C.8 in \cite{diss}) it holds with $\xi_m(j)$ as above (for $r\ge 0$)
\begin{align*}
	&\E\left(\max_{m\, 2^r+1\leq k\leq m\, 2^{r+1}}\left|\sum_{j=1}^k\xi_m(j)\right|\right)^2\leq 4u(m)\,m 2^r\log_2^2\left(4 m 2^r\right)=4 u(m) \,m \, 2^r \,(2+\log_2(m)+r)^2.
\end{align*}

Thus by an application of Markov's inequality
\begin{align*}
	&P\left(\sup_{k> m}\left|\frac{1}{k\sqrt{m}}\sum_{i=1}^m\sum_{j=1}^{k}g_m(Y_i,Y_{j,m}')\right|>\epsilon\right)\le P\left(\sup_{r\geq 0}\frac{1}{m^{3/2}\,2^r}\max_{m2^r+1\leq k\leq m2^{r+1}}\left|\sum_{j=1}^k\xi_m(j)\right|>\epsilon\right)\\
	&\leq \sum_{r\geq 0} P\left(\frac{1}{m^{3/2}\,2^r}\max_{m2^r+1\leq k\leq m2^{r+1}}\left|\sum_{j=1}^k\xi_m(j)\right|>\epsilon\right)\le \frac{4 u(m)}{m^2\,\epsilon^2}\sum_{r\ge 0}\frac{(2+\log_2(m)+r)^2}{2^r}\\
	&=O\left( \frac{u(m)\,\log^2(m)}{m^2} \right)=o(1).
\end{align*}
Additionally, it holds by Assumption~\ref{regw} (iii) 
\begin{align*}
&\sup_{k> m}w(m,k)\frac{k}{\sqrt{m}}\leq \sup_{k> m}\frac{k}{m}\rho\left(\frac km\right)\leq \sup_{t> 1}t\rho\left(t\right)=O(1).
\end{align*}
Together with \eqref{eq_ck_lem2} this yields
\begin{align*}
&\sup_{k> m}w(m,k)\max_{1\leq l\leq k}\left|\frac{1}{m}\sum_{i=1}^m\sum_{j=1}^{l}g_m(Y_i,Y_{j,m}')\right|=O(1)\,\sup_{k> m}\max_{1\leq l\leq k}\left|\frac{1}{k\sqrt{m}}\sum_{i=1}^m\sum_{j=1}^{l}g_m(Y_i,Y_{j,m}')\right|\\
 &\le O(1)\,\left(\max_{1\leq l\leq m}\left|\frac{1}{m^{3/2}}\sum_{i=1}^m\sum_{j=1}^{l}g_m(Y_i,Y_{j,m}')\right|+\sup_{l>m}\left|\frac{1}{l\sqrt{m}}\sum_{i=1}^m\sum_{j=1}^{l}g_m(Y_i,Y_{j,m}')\right|\right)=o_P(1),\end{align*}
completing the proof.
\end{proof}

\subsection{Proofs of Section \ref{sec.asH0}}
More detailed versions of the proofs in this section can be found in \cite{diss}, sections 3.2 and 4.1.

\begin{proof}[Proof of Theorem \ref{as.H0}]
	By Hoeffdings decomposition~\eqref{Hdec} it holds
	\begin{align*}
	&	\Gamma(m,k)=\widetilde{\Gamma}(m,k)+R(m,k),\\
	&\text{with } \widetilde{\Gamma}(m,k)=\sum_{j=m+1}^{m+k}h_2(Y_j)+\frac{k}{m}\sum_{i=1}^{m}h_1(Y_i), \quad R(m,k)=\frac{1}{m}\sum_{i=1}^m\sum_{j=m+1}^{m+k}r(Y_i,Y_j).
	\end{align*}
	Denote by $\widetilde{\Psi}_j(m,k)$, $j=1,2,3$, the respective monitoring schemes based on $\widetilde{\Gamma}$ instead of $\Gamma$ as in Section~\ref{sec.scheme}.

	Lemma~\ref{r.conv.generalg} shows in all three cases, that the remainder term is negligible, i.e.\ 
	\begin{align*}
		\sup_{k\ge 1}w(m,k)|\Psi_j(m,k)	|=\sup_{k\ge 1}w(m,k)|\widetilde{\Psi}_j(m,k)|+o_p(1).
	\end{align*}

	We first derive some preliminary results that will be useful for the proofs of all three assertions.
	By the existence of the limit of $t\rho(t)$ for $t\to\infty$ and the continuity of $\rho$, it holds as $m\to\infty$ 
\begin{align}\label{rhosup}
	\sup_{t>\tau }\left|t\rho(t)-\frac{\lfloor t m\rfloor}{m}\,\rho\left(\frac{\lfloor t m\rfloor}{m}\right)\right|=o(1),\qquad 	\sup_{t>\tau }\left|\rho(t)-\rho\left(\frac{\lfloor t m\rfloor}{m}\right)\right|=o(1). \end{align}

Furthermore, by Assumption~\ref{regass} (ii) and (iv) it holds
\begin{align*}
&\sup_{k>mT}\frac{\sqrt{m}}{k}\sup_{1\leq l\leq k}\left|\sum_{j=m+1}^{m+l}h_2(Y_j)\right|\\
&
\leq
\frac{1}{T}\sup_{1\leq l\leq m}\frac{1}{\sqrt{m}}\left|\sum_{j=m+1}^{m+l}h_2(Y_j)\right|+\frac{1}{\sqrt{T}}\sup_{m\leq l\leq m\sqrt{T}}\frac{\sqrt{m}}{\sqrt{T}m}\left|\sum_{j=m+1}^{m+l}h_2(Y_j)\right|\\*
&\qquad+\frac{1}{T^{1/4}}\sup_{m\sqrt{T}\leq l\leq mT}\frac{T^{1/4}\sqrt{m}}{Tm}\left|\sum_{j=m+1}^{m+l}h_2(Y_j)\right|+\frac{1}{\sqrt{T}}\sup_{l>mT}\frac{\sqrt{Tm}}{l}\left|\sum_{j=m+1}^{m+l}h_2(Y_j)\right|\\
&\leq\frac{1}{T}\sup_{1\leq l\leq m}\frac{1}{\sqrt{m}}\left|\sum_{j=m+1}^{m+l}h_2(Y_j)\right|+\frac{1}{\sqrt{T}}\sup_{l\geq m}\frac{\sqrt{m}}{l}\left|\sum_{j=m+1}^{m+l}h_2(Y_j)\right|\\*
&\qquad+\frac{1}{T^{1/4}}\sup_{l\geq m\sqrt{T}}\frac{T^{1/4}\sqrt{m}}{l}\left|\sum_{j=m+1}^{m+l}h_2(Y_j)\right|+\frac{1}{\sqrt{T}}\sup_{l>mT}\frac{\sqrt{Tm}}{l}\left|\sum_{j=m+1}^{m+l}h_2(Y_j)\right|\\
&=o_P(1)\quad\mbox{as }T\rightarrow\infty\mbox{ uniformly in m.}
\end{align*}
With $\lim_{t\rightarrow\infty}t\rho(t)<\infty$ this yields uniformly in $m$ as $T\to\infty$
\begin{align}\label{eq_new_hr}
&\sup_{k> mT}w(m,k)\sup_{1\leq l\le k}\left|\sum_{j=m+1}^{m+l}h_2(Y_j)\right|
\leq \sup_{t>T}t\rho(t)\sup_{k> mT}\sup_{1\leq l\le k}\frac{\sqrt{m}}{k}\left|\sum_{j=m+1}^{m+l}h_2(Y_j)\right|= o_P(1).
\end{align}

Furthermore, it holds by Assumptions~\ref{regw}(ii) as well as \ref{regass} (ii) and (iii) \begin{align}\label{eq.diff3.pa}
&\sup_{1\leq k\leq \tau m}w(m,k)\sup_{0\leq l\leq k}\left|\sum_{j=m+l+1}^{m+k}h_2(Y_j)+\frac{k-l}{m}\sum_{i=1}^{m}h_1(Y_i)\right|\notag\\
&\leq\sup_{0<t<\tau}\tau^{\gamma}\rho\left(\tau\right)\,\sup_{1\leq k\leq \tau m}\sup_{0\leq l\leq k}\frac{1}{m^{\frac{1}{2}-\gamma}k^{\gamma}}\left|\sum_{j=m+l+1}^{m+k}h_2(Y_j)\right|+\sup_{0<t<\tau}\tau\rho\left(\tau\right)\,
\frac{1}{\sqrt{m}}\left|\sum_{i=1}^{m}h_1(Y_i)\right|\notag\\
&\le \sup_{0<t<\tau}\tau^{\gamma}\rho\left(\tau\right)\,\left(2\,\sup_{1\le k\le m}\frac{1}{m^{\frac{1}{2}-\gamma}k^{\gamma}}\left|\sum_{j=1}^{m+k}h_2(Y_j)\right|+ \frac{1}{\sqrt{m}}\left|\sum_{i=1}^{m}h_1(Y_i)\right|\right)
=o_P(1)
\end{align}
for $\tau\to 0$ (uniformly in $m$).

Let 
\begin{align}\label{eq_wp}
	W_2(t):=\frac{1}{\sigma_2}(\tilde{W}_2(1+t)-\tilde{W}_2(1)),\qquad W_1(t):=\frac{1}{\sigma_1}\tilde{W}_1(t)
\end{align}
for $0< t\leq T$ with $\tilde{W}_j$ as in Assumption~\ref{regass} (ii). Then, $W_1$ and $W_2$ are independent standard Wiener Processes and  we obtain
\begin{align}\label{eq_new_wp}
&\sup_{0< t\leq \tau}\rho(t)\sup_{0\leq s\leq t}\left|\tilde{W}_2(1+t)-\tilde{W}_2(1+s)+(t-s)\tilde{W}_1(1)\right|\notag\\
&=\sup_{0< t\leq \tau}\rho(t)\sup_{0\le s\leq t}\left|\sigma_2\left(W_2(t)-W_2(s)\right)+(t-s)\sigma_1W_1(1)\right|\notag\\
&\leq 2\sigma_2\,\sup_{0< t\leq \tau}t^{\gamma}\rho(t)\,\sup_{0\le s\leq \tau}s^{-\gamma}\left|W_2(s)\right|+\sigma_1\,\tau^{1-\gamma}\sup_{0< t\leq \tau}t^{\gamma}\rho\left(t\right)\left|W_1(1)\right|=o_P(1)
\end{align}
as $\tau\to 0$,
where the last line follows from the modulus of continuity of a Wiener process as well as Assumption~\ref{regw} (ii).

\vspace{2mm}

Now, we are ready to prove the main assertions, where we begin with (i).
First note, that for $0<\tau<T<\infty$ arbitrary but fixed we obtain by \eqref{eq_new_hr}
\begin{align}\label{eq.diff1.gen}
	&\sup_{k>\tau m}\left|w(m,k)\widetilde{\Psi}_1(m,k) -
w(m,k)\left(\sum_{j=m+1}^{m+\min(k,mT)}h_2(Y_j)+\frac{k}{m}\sum_{i=1}^{m}h_1(Y_i)\right)\right|\notag\\
&\leq\sup_{k> mT}w(m,k)\left|\sum_{j=m+1}^{m+mT}h_2(Y_j)\right|+\sup_{k> mT}w(m,k)\left|\sum_{j=m+1}^{m+k}h_2(Y_j)\right|=o_P(1)\end{align}
uniformly in $m$ as $T\to\infty$. Furthermore, by \eqref{rhosup} and Assumption~\ref{regass} (ii) it holds
\begin{align}
	&\sup_{k>\tau m}w(m,k)\left|\sum_{j=m+1}^{m+\min(k,mT)}h_2(Y_j)+\frac{k}{m}\sum_{i=1}^{m}h_1(Y_i)\right|\notag\\
	&=\sup_{t\ge \frac{\lfloor \tau m\rfloor +1}{m}}\rho\left( \frac{\lfloor mt\rfloor}{m} \right)\,\left|\frac{1}{\sqrt{m}}\sum_{j=m+1}^{m+\lfloor m\min(t,T)\rfloor}h_2(Y_j)+\frac{\lfloor m t\rfloor}{m}\,\frac{1}{\sqrt{m}}\sum_{i=1}^{m}h_1(Y_i)\right|\notag\\
	&=\sup_{t\ge \frac{\lfloor \tau m\rfloor +1}{m}}\rho\left(t\right)\,\left|\frac{1}{\sqrt{m}}\sum_{j=m+1}^{m+\lfloor m\min(t,T)\rfloor}h_2(Y_j)+t\,\frac{1}{\sqrt{m}}\sum_{i=1}^{m}h_1(Y_i)\right|+o_P(1)\notag\\
	&\stackrel{\mathcal{D}}{\longrightarrow} \sup_{t>\tau}\rho(t)\left|\tilde{W}_2(1+\min(t,T))-\tilde{W}_2(1)+t\tilde{W}_1(1)\right|.\label{eq.min.gen.final}
\end{align}
Similarly, by the boundedness of $t\rho(t)$ for $t\to\infty$ and the law of iterated logarithm we obtain for $T\to\infty$

\begin{align}\label{eq.diff2.gen}
&\sup_{t>\tau}\left|\rho(t)\left(\tilde{W}_2(1+\min(t,T))-\tilde{W}_2(1)+t\tilde{W}_1(1)\right)-\rho(t)\left(\tilde{W}_2(1+t)-\tilde{W}_2(1)+t\tilde{W}_1(1)\right)\right|\notag\\
&\leq 2\,\sup_{t> T}t\rho(t)\,\sup_{s\ge T}\left|\frac{\tilde{W}_2(1+s)}{s}\right|=o(1)\mbox{ a.s.} 
\end{align}
Combining \eqref{eq.diff3.pa}, \eqref{eq_new_wp}, \eqref{eq.diff1.gen}, \eqref{eq.min.gen.final},  \eqref{eq.diff2.gen} and \eqref{eq.diff2.gen}, for details see Lemma B.2 in \cite{diss}, assertion (i) is obtained with $W_1$ and $W_2$ as in \eqref{eq_wp}.

\vspace{2mm}

For (ii) we get similarly with \eqref{eq_new_hr}  for $0<\tau<T<\infty$ arbitrary but fixed
\begin{align*}
		&\sup_{k>\tau m}w(m,k)\left|\widetilde{\Psi}_2(m,k) -
		w(m,k)\left(\sum_{j=m+\lfloor\min(k,mT)b\rfloor+1}^{m+\min(k,mT)}h_2(Y_j)+\frac{k-\lfloor kb\rfloor}{m}\sum_{i=1}^{m}h_1(Y_i)\right)\right|\\
		&\le \sup_{k> mT}w(m,k)\left|\sum_{j=m+\lfloor mTb\rfloor+1}^{m+mT}h_2(Y_j)\right|+\sup_{k> mT}w(m,k)\left|\sum_{j=m+\lfloor kb\rfloor+1}^{m+k}h_2(Y_j)\right|=o_P(1)
\end{align*}
for $T\to \infty$ (uniformly in $m$).
Using double Gaussian brackets in the lower index of the sum in the next proof as opposed to a single Gaussian bracket can result in a by at most one shifted lower index. However, because the maximum over all summands is of order $o_P(\sqrt{m})$ as can e.g. be seen by the functional central limit theorem in Assumption~\ref{regass} (ii), this difference is neglible in the below context; for details we refer to (4.13) in \cite{diss}. Similarly, the double bracket in the linear part of the second term can be replaced by a single bracket, because the sum there is of order $O_P(\sqrt{m})$.
Thus we get by\eqref{rhosup}
\begin{align*}
&	\sup_{k>\tau m}w(m,k)\left(\sum_{j=m+\lfloor\min(k,mT)b\rfloor+1}^{m+\min(k,mT)}h_2(Y_j)+\frac{k-[kb]}{m}\sum_{i=1}^{m}h_1(Y_i)\right)\\
&=\sup_{t\ge \frac{\lfloor \tau m\rfloor +1}{m}} \rho\left( \frac{\lfloor mt\rfloor}{m} \right)\, \left(\frac{1}{\sqrt{m}}\sum_{j=m+\lfloor \min(\lfloor m t\rfloor,mT)b]+1}^{m+\lfloor m\, \min(t,T)\rfloor}h_2(Y_j)+\frac{\lfloor m t \rfloor-\lfloor\lfloor m t\rfloor b\rfloor}{m}\sum_{i=1}^{m}h_1(Y_i)\right)\\
&=\sup_{t\ge \frac{\lfloor \tau m\rfloor +1}{m}} \rho\left( \frac{\lfloor mt\rfloor}{m} \right)\, \left(\frac{1}{\sqrt{m}}\sum_{j=m+\lfloor m\,\min(t,T)\,b\rfloor+1}^{m+\lfloor m\, \min(t,T)\rfloor}h_2(Y_j)+(1-b)\,\frac{\lfloor m t \rfloor}{m}\sum_{i=1}^{m}h_1(Y_i)\right)+o_P(1)\\
&=\sup_{t\ge \frac{\lfloor \tau m\rfloor +1}{m}} \rho\left( t\right)\, \left(\frac{1}{\sqrt{m}}\sum_{j=m+\lfloor m\,\min(t,T)\,b\rfloor+1}^{m+\lfloor m\, \min(t,T)\rfloor}h_2(Y_j)+(1-b)\,t\,\sum_{i=1}^{m}h_1(Y_i)\right)+o_P(1)\\
&\stackrel{\mathcal{D}}{\longrightarrow} \sup_{t>\tau}\rho(t)\left|\tilde{W}_2(1+\min(t,T))-\tilde{W}_2(1+\min(t,T)b)+t(1-b)\,\tilde{W}_1(1)\right|.
\end{align*}
By the law of iterated logarithm it follows as $T\to\infty$
\begin{align*}
	&\sup_{t>\tau}\left|\rho(t)\left(\tilde{W}_2(1+\min(t,T))-\tilde{W}_2(1+\min(t,T)\,b)+t(1-b)\tilde{W}_1(1)\right)\right.\notag\\*
	&\qquad\left.-\rho(t)\left(\tilde{W}_2(1+t)-\tilde{W}_2(1+tb)+t(1-b)\tilde{W}_1(1)\right)\right|\notag\\
&\leq 4\sup_{t\ge Tb}t\rho(t)\sup_{s\ge Tb}\left|\frac{\tilde{W}_2(1+s)}{s}\right|=o_P(1).
\end{align*}
This completes the proof of (ii) as in (i).

\vspace{2mm}

For the proof of (iii) by \eqref{eq_new_hr} for $0<\tau<T<\infty$ arbitrary but fixed
\begin{align*}
	&\sup_{k>\tau m}w(m,k)\left|\widetilde{\Psi}_3(m,k)-\sup_{1\leq l\leq k}\left|\sum_{j=m+\min(l,mT)+1}^{m+\min(k,mT)}h_2(Y_j)+\frac{k-l}{m}\sum_{i=1}^{m}h_1(Y_i)\right|\right|\\
&\leq\sup_{k> mT}w(m,k)\left|\sum_{j=m+mT+1}^{m+k}h_2(Y_j)\right|+\sup_{k> mT}w(m,k)\sup_{mT\leq l\leq k}\left|\sum_{j=m+l+1}^{m+k}h_2(Y_j)\right|=o_P(1)
\end{align*}
for $T\to \infty$ (uniformly in $m$).
By \eqref{rhosup} we get
\begin{align*}
	&\sup_{t>\tau}\sup_{0<s\le t}\left|\rho\left( \frac{\lfloor mt\rfloor}{m} \right)\frac{\lfloor  mt\rfloor-\lfloor ms \rfloor}{m}- \rho(t) (t-s)\right|\\
	&\le \left(1+\sup_{0<s\le t} \frac{\lfloor ms \rfloor}{\lfloor m t\rfloor}\right)\,\sup_{t>\tau}\left|\rho\left( \frac{\lfloor mt\rfloor}{m} \right)\frac{\lfloor mt\rfloor}{m}-\rho(t)t\right|+\sup_{t>\tau}\rho(t)t \,\sup_{0<s\le t} \left| \frac{s}{t}-\frac{\lfloor m s \rfloor}{\lfloor m t\rfloor} \right|\\
	&\le o(1)+O\left( \frac{1}{m\tau} \right)=o(1),
\end{align*}
\end{proof}

\begin{proof}[Proof of Corollary \ref{simple1}]
Comparing the covariance structures, we first obtain
\begin{align}\label{simpleW}
\left\{W_1(t)+tW_2(1):t\geq 0\right\}\stackrel{\mathcal{D}}{=}\left\{(1+t)W\left(\frac{t}{1+t}\right):t> 0\right\}.
\end{align}
Thus, we obtain (a) (i) as
 \begin{align*}
&\sup_{t>0}\rho(t)\left|W_2(t)+tW_1(1)\right|\stackrel{\mathcal{D}}{=}\sup_{t>0}\rho(t)(1+t)W\left(\frac{t}{1+t}\right)\stackrel{\mathcal{D}}{=} \sup_{0<s<1}\rho\left(\frac{s}{1-s}\right)\frac{|W(s)|}{1-s}.
\end{align*}
Similarly, for (a) (ii)
\begin{align*}
&\sup_{t>0}\rho(t)\left|W_2(t)-W_2(th)+t(1-h)W_1(1)\right|
\stackrel{\mathcal{D}}{=}\sup_{t>0}\rho(t)\left|(1+t)W\left(\frac{t}{1+t}\right)-(1+th)W\left(\frac{th}{1+th}\right)\right|\\
&\stackrel{\mathcal{D}}{=}\sup_{0<s<1}\rho\left(\frac{s}{1-s}\right)\left|\frac{W(s)}{1-s}-(1-s(1-h))\frac{W\left(\frac{sh}{1-s(1-h)}\right)}{1-s}\right|,
\end{align*}
as well as for (a) (iii)
\begin{align*}
&\sup_{t>0}\rho(t)\sup_{0<s\leq t}\left|W_2(t)-W_2(s)+(t-s)W_1(1)\right|\\
&
\stackrel{\mathcal{D}}{=}\sup_{t>0}\rho(t)\sup_{0<s\leq t}\left|(1+t)W\left(\frac{t}{1+t}\right)-(1+s)W\left(\frac{s}{1+s}\right)\right|\\
&\stackrel{\mathcal{D}}{=}\sup_{0<\tilde{t}<1}\rho\left(\frac{\tilde{t}}{1-\tilde{t}}\right)\sup_{0<\tilde{s}\leq \tilde{t}}\left|\frac{W\left(\tilde{t}\right)}{1-\tilde{t}}-\frac{W\left(\tilde{s}\right)}{1-\tilde{s}}\right|.
\end{align*}
Assertion (a) now follows by Theorem~\ref{as.H0} and Slutzkys Lemma.

Concerning (b) comparing covariances yields
\begin{align*}\{\sigma_1\left(\sigma_2W_2(t)+t\sigma_1W_1(1)\right):t>0\}\stackrel{\mathcal{D}}{=}\left\{(\sigma_2^2+\sigma_1^2t)\,W\left(\frac{\sigma_1^2t}{\sigma_2^2+\sigma_1^2t}\right):t>0\right\}.
\end{align*}
For (b) (i) this yields
\begin{align*}
	\sup_{t>0}\frac{\left|\sigma_1(\sigma_2W_2(t)+t\sigma_1W_1(1))\right|}{\sigma_2^2+\sigma_1^2t}\stackrel{\mathcal{D}}{=}\sup_{t>0}\left|W\left(\frac{\sigma_1^2t}{\sigma_2^2+\sigma_1^2t}\right)\right|=\sup_{0<s<1}|W(s)|,
\end{align*}
as well as for (b) (ii)
\begin{align*}
	&\sup_{t>0}\frac{\sigma_1\left|\sigma_2(W_2(t)-W_2(th))+t(1-h)\sigma_1W_1(1)\right|}{\sigma_2^2+\sigma_1^2t}\\
&\stackrel{\mathcal{D}}{=}\sup_{t>0}\left|W\left(\frac{\sigma_1^2t}{\sigma_2^2+\sigma_1^2t}\right)-\frac{\sigma_2^2+\sigma_1^2th}{\sigma_2^2+\sigma_1^2t}W\left(\frac{\sigma_1^2th}{\sigma_2^2+\sigma_1^2th}\right)\right|\\
&\stackrel{\mathcal{D}}{=}\sup_{0<s<1}\left|W(s)-\frac{1-s+h}{1-s+1}W\left(\frac{h}{1-s+h}\right)\right|,
\end{align*}
and (b) (iii)
\begin{align*}
	&\sup_{t>0}\sup_{0<s\leq t}\frac{\sigma_1\left|\sigma_2(W_2(t)-W_2(s))+(t-s)\sigma_1W_1(1)\right|}{\sigma_2^2+\sigma_1^2t}\\
&\stackrel{\mathcal{D}}{=}\sup_{t>0}\left|W\left(\frac{\sigma_1^2t}{\sigma_2^2+\sigma_1^2t}\right)-\frac{\sigma_2^2+\sigma_1^2s}{\sigma_2^2+\sigma_1^2t}W\left(\frac{\sigma_1^2s}{\sigma_2^2+\sigma_1^2s}\right)\right|
\stackrel{\mathcal{D}}{=}\sup_{0<\tilde{t}<1}\sup_{0<\tilde{s}\leq \tilde{t}}\left|W(\tilde{t})-\frac{1+(1-\tilde{s})^{-1}}{1+(1-\tilde{t})^{-1}}W(\tilde{s})\right|.
\end{align*}
Assertion (b) follows from Theorem~\ref{as.H0} by
\begin{align*}
&\sup_{k\geq 1}\frac{\hat{\sigma}_{m,1}}{\sqrt{m}\left(\hat{\sigma}_{2,m}^2+\hat{\sigma}_{1,m}^2\frac{k}{m}\right)}\left|\Gamma_i(m,k)\right|\\
&
=
\sup_{k\geq 1}
\frac{\sigma_{1}}{\sqrt{m}\left(\sigma_{2}^2+\sigma_{1}^2\frac{k}{m}\right)}
\left|\Gamma_i(m,k)\right|+
\sup_{k\ge 1}\left|
\frac{\hat{\sigma}_{1,m}}{\sigma_1}\, \frac{\sigma_2^2+\sigma_1^2\frac{k}{m}}{\hat{\sigma}_{2,m}^2+\hat{\sigma}_{1,m}^2\frac km} -1
\right|\cdot O_P\left(1\right),
\end{align*}
which yields the assertion by
\begin{align*}
	&\sup_{k\ge 1}\left| \frac{\hat{\sigma}_{1,m}}{\sigma_1}\, \frac{\sigma_2^2+\sigma_1^2\frac{k}{m}}{\hat{\sigma}_{2,m}^2+\hat{\sigma}_{1,m}^2\frac km} -1\right|
	=\sup_{k\ge 1}\left|\frac{\hat{\sigma}_{1,m}\sigma_2^2-\sigma_1\hat{\sigma}_{2,m}^2}{\sigma_1\left( \hat{\sigma}_{2,m}^2+\hat{\sigma}_{1,m}^2\frac km \right)}+\frac{\left( \hat{\sigma}_{1,m}\sigma_1-\hat{\sigma}_{1,m}^2 \right)\frac{k}{m}}{\hat{\sigma}_{2,m}^2+\hat{\sigma}_{1,m}^2\frac km }
		\right|\\
		&\le \frac{|\hat{\sigma}_{1,m}\sigma_2^2-\sigma_1\hat{\sigma}_{2,m}^2|}{\sigma_1\, \hat{\sigma}_{2,m}^2}+\frac{\left| \hat{\sigma}_{1,m}\sigma_1-\hat{\sigma}_{1,m}^2 \right|}{\hat{\sigma}_{1,m}^2 }=o_p(1).
\end{align*}
\end{proof}


\subsection{Proofs of Section \ref{sec.asH1}}
We are now ready to prove Theorem~\ref{TheoremH1w2}, i.e.\ that the monitoring statistics will eventually stop with probability one under alternatives and that the corresponding sequential test has asymptotic power one. More detailed versions of the proofs in this section can be found in \cite{diss}, sections 3.3 and 4.2.

\begin{proof}[Proof of Theorem \ref{TheoremH1w2}]
	Let $\tilde{k}>k^*$ with $\tilde{k}=\tilde{k}_m>\nu m$, $\nu>0$.
Then, it holds with \eqref{hoeffh1}
\begin{align}\label{GammaH1tild}
\Gamma\left(m,\tilde{k}\right)
=& \Gamma(m,k^*)+\frac{1}{m}\sum_{i=1}^m\sum_{j=m+k^*+1}^{m+\tilde{k}}r_m^*(Y_i,Z_{j,m})
+\sum_{j=m+k^*+1}^{m+\tilde{k}}h^*_{2,m}(Z_{j,m})\notag\\*
&+\frac{\tilde{k}-k^*}{m}\sum_{i=1}^m h^*_1(Y_i)+(\tilde{k}-k^*)\Delta_m
\end{align}
and corresponding decompositions of the monitoring statistics $\Psi_j(m,\tilde{k})$.
For each situation discussed, we will pick appropriate $\tilde{k}$, such that the signal term $(\tilde{k}-k^*)\Delta_m$ dominates the other four terms.

In fact, by Assumption~\ref{regw} it holds
\begin{align}
	\sup_{t\ge \nu}(1+t)\rho(t)<\infty.
	\label{eq_rho_H1}
\end{align}
By this and Theorem~\ref{as.H0} (applied with weight function~\eqref{exw} with $\gamma=0$) we get
\begin{align}\label{eq_psi_H1}
	w(m,\tilde{k})|\Psi_j(m,k^*)|\le \left( 1+\frac{\tilde{k}}{m} \right)\,\rho\left(\frac{\tilde{k}}{m}\right)\; \frac{1}{\sqrt{m}}\,\frac{1}{1+\frac{k^*}{m}}|\Psi_j(m,k^*)|=O_P(1).
\end{align}
By  Assumption~\ref{regassH1} (i) it holds for any $k^*\le \utilde{k}\le \tilde{k}$ 
\begin{align}
	&w(m,\tilde{k})\left|\frac{1}{m}\sum_{i=1}^m\sum_{j=m+\utilde{k}+1}^{m+\tilde{k}}r_m^*(Y_i,Z_{j,m})\right|\notag\\
	&=\frac{\tilde{k}}{m}\rho\left( \frac{\tilde{k}}{m} \right)\;\frac{1}{\tilde{k}\,\sqrt{m}}\left|\sum_{i=1}^m\sum_{j=m+\utilde{k}+1}^{m+\tilde{k}}r_m^*(Y_i,Z_{j,m})\right|=O_P\left( 1\right).
	\label{eq_rem_H1}
\end{align}
Similarly, by Assumption~\ref{regassH1} 
\begin{align}
	w(m,\tilde{k})\left| \sum_{j=m+\utilde{k}+1}^{m+\tilde{k}}h^*_{2,m}(Z_{j,m}) \right|=O_P(1),\qquad	w(m,\tilde{k})\left|\frac{\tilde{k}-\utilde{k}}{m}\sum_{i=1}^m h^*_1(Y_i)\right|=O_P(1).
	\label{eq_sum_H1}
\end{align}

We are now ready to prove that the signal term with appropriate choice of $\tilde{k}$ converges to $\infty$ guaranteeing asymptotic power one.
To this end, consider the first the situation of late changes with $\frac{k^*}{m}\rightarrow\infty$. For the CUSUM and Page-CUSUM monitoring scheme $\Psi_j$, $j=1,3$, choose $\tilde{k}=2k^*$, such that with \eqref{GammaH1tild} -- \eqref{eq_sum_H1} (for $\utilde{k}=k^*$) it holds by Assumption~\ref{regassH1} (i) and $\sqrt{m}|\Delta_m|\to \infty$
\begin{align*}
	&\sup_{k\ge 1}w(m,k)|\Psi_j(m,k)|\ge w(m,\tilde{k})|\Psi_j(m,\tilde{k})|\\
	&\ge w(m,\tilde{k})(\tilde{k}-k^*)|\Delta_m|+O_P(1)=\frac{k^*}{m}\rho\left( 2\frac{k^*}{m} \right)\; \sqrt{m}|\Delta_m|+O_P(1)\overset{P}{\longrightarrow} \infty.
\end{align*}
Similarly, for the mMOSUM-monitoring scheme $j=2$ consider $\tilde{k}=\lfloor k^*/b\rfloor+1$ such that $\lfloor b\tilde{k}\rfloor\ge k^*$ and  $(\tilde{k}-\lfloor \tilde{k}b\rfloor)/\tilde{k}\ge 1-b$. Then,  we obtain with $\utilde{k}=\lfloor b \tilde{k}\rfloor$ by $\tilde{k}/m\to \infty$
\begin{align*}
	&\sup_{k\ge 1}w(m,k)|\Psi_2(m,k)|\ge w(m,\tilde{k})|\Psi_2(m,\tilde{k})|\ge w(m,\tilde{k})\,(\tilde{k}-\lfloor b\tilde{k}\rfloor)\,|\Delta_m|+O_P(1)\\
	&\ge \frac{\tilde{k}}{m}\rho\left( \frac{\tilde{k}}{m} \right)\,(1-b)\,\sqrt{m}|\Delta_m|+O_P(1)\overset{P}{\longrightarrow} \infty.
\end{align*}
For early changes with $k^*/m=O(1)$ we choose for CUSUM and Page-CUSUM $\tilde{k}=\lfloor t_0 m\rfloor$ with $t_0$ as in Assumption~\ref{altrho} (ii). Then, as above we get for $j=1,3$
\begin{align*}
	&\sup_{k\ge 1}w(m,k)|\Psi_j(m,k)|\ge w(m,\tilde{k})|\Psi_j(m,\tilde{k})|
	=w(m,\tilde{k})\,(\tilde{k}-k^*)|\Delta_m|+O_P(1)\\
	&\ge\rho\left( \frac{ \lfloor t_0 m\rfloor}{m} \right)\,\left( \frac{\lfloor t_0 m\rfloor}{m}-\nu \right)\,\sqrt{m}|\Delta_m|+O_P(1)\overset{P}{\longrightarrow} \infty.
\end{align*}
Finally, for the modified MOSUM ($j=2$) it holds with $\tilde{k}=\lfloor m t_0/b\rfloor +1$ with $t_0$ as in Assumption~\ref{altrho}~(iii)
\begin{align*}
		&\sup_{k\ge 1}w(m,k)|\Psi_2(m,k)|\ge w(m,\tilde{k})\,(\tilde{k}-\lfloor b\tilde{k}\rfloor)\,|\Delta_m|+O_P(1)\\
		&\ge \rho\left( \frac{\lfloor m t_0/b\rfloor +1}{m} \right)\,\left(t_0\left( \frac{1}{b}-1\right)-\frac{b}{m} \right)\;\sqrt{m}|\Delta_m|+O_P(1)\overset{P}{\longrightarrow} \infty,
\end{align*}
completing the proof.
\end{proof}

\subsection{Proofs of Section \ref{sec.exass}}
\begin{proof}[Proof of Lemma \ref{ass.ex}]
	First consider independent $\{Y_i\}$ as in (a). For (i), note that for $1\le  i_1\neq i_2\le m$ and $j\in\{m+1,\ldots,m+k\}$ it holds by 
\eqref{hcenter} 
\begin{align*}
&\E\left(h(Y_{i_1},Y_j)h(Y_{i_2},Y_j)\right)=\int\int h(y_{i_2},y_j)dF(y_{i_2})\int h(y_{i_1},y_j)dF(y_{i_1})dF(y_j)\notag\\
&=\int(h_2(y_j)+\theta)^2dF(y_j)=\E((h_2(Y_j)+\theta)^2)=\E(h_2(Y_j)^2)+\theta^2,
\end{align*}
as well as
\begin{align*}
&\E\left(h(Y_{i_1},Y_j)h_2(Y_j)\right)=\int h_2(y_j)\int h(y_{i_1},y_j)dF(y_{i_1})dF(y_j)\\
&=\int h_2(y_j)^2dF(y_j)+\theta \int h_2(y_j)dF(y_j)=\E(h_2(Y_j)^2).
\end{align*}
Consequently, by the independence we get
\begin{align*}
&\Cov\left(r(Y_{i_1},Y_j),r(Y_{i_2},Y_j)\right)\notag\\
&=\Cov\left(h(Y_{i_1},Y_j)-h_1(Y_{i_1})-h_2(Y_j)-\theta,h(Y_{i_2},Y_j)-h_1(Y_{i_2})-h_2(Y_j)-\theta\right)\notag\\
&=\E\left(h(Y_{i_1},Y_j)h(Y_{i_2},Y_j) \right)-\theta^2-\E\left(h(Y_{i_1},Y_j)h_2(Y_j)\right)-\E\left(h_2(Y_j)h(Y_{i_2},Y_j) \right)+\E\left( h_2(Y_j)^2\right)\\
&=0.
\end{align*}

For $m<j_1\neq j_2\le m+k,$ and $i\in\{1,\ldots,m\}$ it follows analogously that $r(Y_i,Y_{j_1})$ and $r(Y_i,Y_{j_2})$ are uncorrelated. In the case of $i_1\neq i_2\neq j_1\neq j_2$ $r(Y_{i_1},Y_{j_1})$ and $r(Y_{i_2},Y_{j_2})$ are independent.
Hence,  we obtain
\begin{align*}
\E\left(\left|\sum_{i=1}^m\sum_{j=k_1}^{k_2}r(Y_i,Y_j)\right|^2\right)=\Var\left(\sum_{i=1}^m\sum_{j=k_1}^{k_2}r(Y_i,Y_j)\right)=\sigma_r^2m(k_2-k_1+1)
\end{align*}
with $\sigma_r^2=\Var \left(r(Y_1,Y_2)\right)<\infty$ due to \eqref{ass.sym}. Hence, (i) follows.
\vspace{2mm}

Assertion (a)(ii) follows  from the 2-dimensional version of Donsker's Theorem (see Theorem 1.1. in \cite{Einmahl2009}).
\vspace{2mm}

Assertions (a)(iii) and (iv) follow from the   H\'ajek-R\'enyi inequality (see \cite{hajek}) for i.i.d.\ random variables with variance $\sigma_2^2=\Var(h_2(Y_1))$,i.e.\ 
\begin{align*}
&P\left(\sup_{1\leq k\leq m}\frac{1}{m^{\frac{1}{2}-\alpha}k^{\alpha}}\left|\sum_{j=m+1}^{m+k}h_2(Y_j)\right|>C\right)\leq\frac{\sigma_2^2}{C^2}m^{2\alpha-1}\sum_{k=1}^m\frac{1}{k^{2\alpha}}\le \frac{\sigma_2^2}{C^2}\frac{1}{1-2\alpha}
\end{align*}
as well as
\begin{align*}
	&P\left(\sup_{k>k_m}\frac{\sqrt{k_m}}{k}\left|\sum_{j=m+1}^{m+k}h_2(Y_j)\right|>C\right)\leq\frac{\sigma_2^2}{C^2}\left(1+k_m\sum_{k=k_m+1}^{\infty}\frac{1}{k^2}\right)\le 2\frac{\sigma_2^2}{C^2}.\end{align*}
\vspace{2mm}

In the dependent situation as in (b), assertion (i) follows  analogously to Lemma 1 in \cite{Dehl} by showing that there exists a constant C such that
$$\E\left(\left|\sum_{i=1}^m\sum_{j=k_1}^{k_2}r(Y_i,Y_j)\right|^2\right)\leq Cm(k_2-k_1+1).$$
\vspace{2mm}

For assertion (b)(ii)  Lemma 2.15 in \cite{borovkova2001limit} yields that $h_1(\cdot)$ and $h_2(\cdot)$ are also 1-continuous functions. Furthermore, they are bounded and centered such that with (\ref{coeff.cond}) the functional central limit theorem is obtained by Proposition D.4 in \cite{diss}.
\vspace{2mm}

For assertions (b)(iii) and (iv), first note that by Proposition 2.11 in \cite{borovkova2001limit}, $\{h_2(Y_i):i\in\mathbb{Z}\}$ is also a 1-approximating functional of an absolutely regular process with approximating constants
$$ a_k'=\Phi(\sqrt{2a_k})+C\sqrt{2a_k}.$$ Hence, by (\ref{coeff.cond}) the assumption a) of Lemma 2.24 in \cite{borovkova2001limit} is fulfilled for $\{h_2(Y_i):i\in\mathbb{Z}\}$. Consequently, it follows with Theorem B.3. in \cite{kdiss} and the stationarity that there exists a constant $A\geq 4$ such that
\begin{align*}
&\E\left(\sup_{1\leq k\leq m}\frac{1}{m^{\frac{1}{2}-\alpha}k^{\alpha}}\left|\sum_{j=m+1}^{m+k}h_2(Y_j)\right|^4\right)\leq CA\frac{1}{m^{2-4\alpha}}\sum_{k=1}^m k^{1-4\alpha}\leq\frac{1}{2-4\alpha}CA
\end{align*}
as well as
\begin{align*}
&\E\left(\sup_{k>k_m}\frac{\sqrt{k_m}}{k}\left|\sum_{j=m+1}^{m+k}h_2(Y_j)\right|^4\right)\le CA\left(\sum_{k=1}^{k_m}\frac{k}{k_m^2}+\sum_{k=k_m+1}^{\infty}\frac{k_m^2}{k^3}\right)\le \frac 32 CA.
\end{align*}
\end{proof}

\subsection*{Acknowledgements}
This work was supported by the Deutsche Forschungsgemeinschaft (DFG, German Research Foundation) - 314838170, GRK 2297 MathCoRe.
\bibliography{BIB}

\end{document}